\documentclass[11pt]{amsart}

\usepackage{amssymb,amsmath,amsthm,verbatim,enumitem,hyperref}
\usepackage{fullpage}
\usepackage{xcolor}
\usepackage{url}
\usepackage{mathrsfs}
\usepackage[all]{xy}
  \SelectTips{cm}{10}
  \everyxy={<2.5em,0em>:}
\usepackage{tikz}

\theoremstyle{plain}
  \newtheorem{lemma}[equation]{Lemma}
  \newtheorem{proposition}[equation]{Proposition}
  \newtheorem{theorem}[equation]{Theorem}
  \newtheorem{corollary}[equation]{Corollary} 
   
    \newtheorem{question}[equation]{Question}

\theoremstyle{definition}
  \newtheorem{definition}[equation]{Definition}

\theoremstyle{remark}
  \newtheorem{remark}[equation]{Remark}

\renewcommand{\thesection}{\arabic{section}}
\renewcommand{\theequation}{\thesection.\arabic{equation}}

 \DeclareFontFamily{U}{manual}{}
 \DeclareFontShape{U}{manual}{m}{n}{ <->  manfnt }{}
 \newcommand{\manfntsymbol}[1]{%
    {\fontencoding{U}\fontfamily{manual}\selectfont\symbol{#1}}}

\makeatletter
   \@addtoreset{section}{part}
   \@addtoreset{equation}{section}
   \@addtoreset{footnote}{section}

    {\hspace*{\fill}$\lrcorner$\endgraf\endgroup\end{trivlist}}
 \newenvironment{example}[1][]{
   \refstepcounter{equation}
   \begin{proof}[Example~\theequation%
   \@ifnotempty{#1}{ (#1)}.]
   }
  {\end{proof}}
\makeatother

  \DeclareFontFamily{OT1}{pzc}{}
  \DeclareFontShape{OT1}{pzc}{m}{it}{<-> s * [1.100] pzcmi7t}{}
  \DeclareMathAlphabet{\mathpzc}{OT1}{pzc}{m}{it}

\newif\ifhascomments \hascommentstrue
\ifhascomments
\else
\fi

\renewcommand{\AA}{\mathbb{A}}

\DeclareMathOperator{\Aut}{\ensuremath{\mathcal{A}\kern-.125em\mathpzc{ut}}}

\DeclareMathOperator{\diag}{diag}

\newcommand{\CC}{\mathbb C}

\renewcommand{\emptyset}{\varnothing}

\DeclareMathOperator{\Endo}{\ensuremath{\mathcal{E}\kern-.125em\mathpzc{nd}}}

\newcommand{\cC}{\mathfrak c}

\newcommand{\fh}{\mathfrak h}

\renewcommand{\sl}{\mathfrak{sl}}
\newcommand{\so}{\mathfrak{so}}
\renewcommand{\sp}{\mathfrak{sp}}

\newcommand{\GG}{\mathbb G}

\DeclareMathOperator{\GL}{GL}

\DeclareMathOperator{\Hom}{\ensuremath{\mathcal{H}\kern-.125em\mathpzc{om}}}

\newcommand{\NN}{\mathbb N}

\newcommand{\ps}{\mathrm{ps}}
\newcommand{\pr}{\mathrm{pr}}

\newcommand{\QQ}{\mathbb Q}
\newcommand{\RR}{\mathbb R}

\DeclareMathOperator{\rk}{rk}

\renewcommand{\setminus}{\smallsetminus}

\DeclareMathOperator{\SL}{SL}
\DeclareMathOperator{\Sp}{Sp}
\DeclareMathOperator{\SO}{SO}

\DeclareMathOperator{\Spin}{Spin}
\DeclareMathOperator{\spec}{Spec}
\DeclareMathOperator{\stab}{Stab}

\DeclareMathOperator{\Sym}{Sym}

\newcommand{\X}{\mathcal{X}}

\newcommand{\ZZ}{\mathbb{Z}}

 \def\ari[#1]{\ar@{^(->}[#1]}
 \def\are[#1]{\ar[#1]^{\txt{\'et}}}
 \def\areh[#1]{\ar[#1]|{\txt{$H$-eq}}^{\txt{\'et}}}
 \def\ars[#1]{\ar@{->>}[#1]}
 \newcommand{\dplus}{\ar@{}[d]|{\mbox{$\oplus$}}}
 \newcommand{\dtimes}{\ar@{}[d]|{\mbox{$\times$}}}

\DeclareMathOperator{\codim}{codim}

\DeclareMathOperator{\sss}{sss}

\DeclareMathOperator{\Pic}{Pic}

\newcommand{\fc}{\frak{c}}

\newcommand{\fg}{\frak{g}}
\newcommand{\cX}{\mathcal{X}}

\makeatletter
\newcommand{\extp}{\@ifnextchar^\@extp{\@extp^{\,}}}
\def\@extp^#1{\mathop{\bigwedge\nolimits^{\!#1}}}
\makeatother

\newcommand{\sslash}{\hspace{-0.1cm}\mathbin{/}\hspace{-0.07cm}}

\title{On a smoothness characterization for good moduli spaces}

\author{Dan Edidin}
\address{Department of Mathematics, University of Missouri, Columbia MO 65211}
\email{edidind@missouri.edu}

\author{Matthew Satriano}
\address{University of Waterloo \\
Department of Pure Mathematics \\
Waterloo, Ontario \\
Canada  N2L 3G1}

\email{msatrian@uwaterloo.ca}

\author{Spencer Whitehead}
\address{Duke University \\
Department of Mathematics \\
Durham, NC 27708}

\email{spencer.whitehead@duke.edu}

\thanks{The first author was supported by Simons Collaboration Grant
  315460. The second author was partially supported by a Discovery Grant from the National Science and Engineering Board of Canada as well as a Mathematics Faculty Research Chair from the University of Waterloo. 
  The third author was partially supported by an Undergraduate Student Research Award from the National Science and Engineering Board of Canada.}

\begin{document}

\begin{abstract}
  Let $\X$ be a smooth Artin stack with properly stable good moduli space $\X \stackrel{\pi} \to X$. The purpose of this paper is to prove that a simple geometric criterion can often characterize when the moduli space $X$ is smooth and the morphism $\pi$ is flat.
\end{abstract}

\maketitle

%

\part{Main Results}
\label{part:general-facts}

\section{Introduction}

Let $K$ be an algebraically closed field of characteristic $0$ and
$\X$ be a $K$-smooth Deligne--Mumford stack with coarse space
$p\colon\X\to X$. Applying the Purity of the Branch Locus Theorem to the proper quasi-finite morphism $p$ yields a necessary condition for $X$ to be smooth:~the branch locus of $p$ must be pure of codimension-one; here the branch locus is the complement of the largest
open set $U \subset X$ over which $p$ is \'etale.
On the other hand, the beautiful theorem of Chevalley--Shephard--Todd gives a simple group-theoretic criterion which is sufficient
for determining when $X$ is
smooth. Specifically, if $x\in\X(K)$ with stabilizer group $G_x$, then
$X$ is smooth at $p(x)$ if and only if the $G_x$-action on the
tangent space $T_{X,x}$ is generated by pseudo-reflections, i.e.~$G$ is generated
by elements ~$g\in G_x$ whose fixed locus is a hyperplane. Whenever $X$ is smooth, $p$ is automatically flat.

For smooth Artin stacks we consider the following situation analogous to
the Deligne--Mumford setting.
Let $\X$ be a smooth Artin stack with properly stable good moduli space 
$p\colon\X\to X$; this means that there is a dense set of points $x$ in $\cX$ which have $0$-dimensional stabilizer 
and are also $p$-saturated, i.e.,~$p^{-1}p(x) = x$ \cite[Definition 2.5]{EdRy:21}.
In this case the good moduli space morphism is not separated but shares some properties of a proper quasi-finite morphism: it is universally closed and if $x\in X$ is a closed point, then there is a unique closed point in the fiber of $p$ over $x$. Such $\X$ arise
naturally in the context of GIT, e.g.~if $G$ is a reductive group with
properly stable action on a variety $U$ and if $U^{ss}$ denotes the
semistable locus, then $p\colon[U^{ss}/G]\to U\sslash G$ is a properly stable good moduli space.

Despite the fact that smooth Artin stacks with properly stable good moduli spaces are analogous to Deligne--Mumford stacks, there are no general necessary and sufficient criteria to determine when $X$ is smooth and $p$ is flat. Indeed, one cannot invoke the Purity of the Branch Locus Theorem as $p$ is not proper quasi-finite, and there is no known analogue of the Chevalley--Shephard--Todd Theorem since smooth Artin stacks do not have tangent bundles. 

The starting point for this paper is to instead take a GIT point of view. By
\cite[Theorem 4.12]{AHR:20} and \cite{Lun:73}, at a closed
point $x$ of $\X$, the map $p$ is \'etale locally isomorphic to $[V/G_x] \to V/G$ for some representation $V$ of the stabilizer group $G_x$. Thus, the problem of determining when $X$ is smooth and $p$ is flat reduces to the case where $\X=[V/G]$ and $V$ is a representation of a linearly reductive group $G$. A natural analogue of the branch locus is then the image in $X$ of points in $V$ which have a {\em positive dimensional stabilizer group}. This is exactly the image of the GIT strictly semi-stable points of $X$. Inspired by the Purity of the Branch Locus Theorem, a na\"{i}ve guess is that the following condition is necessary for $X$ to be smooth:
\begin{equation}
  \tag{$\star$}\label{item.pure}
  \parbox{\dimexpr\linewidth-4em}{%
    \strut
The image of the strictly semi-stable points must be of pure codimension-one.
    \strut
  }
\end{equation}

The main results of this paper imply that condition \eqref{item.pure}
goes a long way toward determining when $V/G$ is smooth and $p$ is flat. Specifically, we prove
for irreducible representations
of simple groups, condition \eqref{item.pure} is both necessary and sufficient
for $V/G$ to be smooth and $[V/G] \to V/G$ to be flat.
In addition we show that when $G$ is a torus, a slight strengthening of condition \eqref{item.pure}
is necessary and sufficient to characterize when $V/G$ is smooth and $[V/G] \to V/G$ is flat.

To state our results precisely we 
introduce Definition \ref{def:V-is-pure} after recalling some basic notions.

\begin{definition}
  Let $V$ be a representation of a reductive group $G$. A vector
  $v \in V$ is {\em $G$-stable}  if
  $Gv$ is closed and $v$ is not contained in the closure of any other
  orbit. A vector $v \in V$ is {\em $G$-properly stable} if $v$ is stable and $\dim Gv = \dim G$.

  A representation $V$ is stable (resp.~properly stable)
  if it contains a stable (resp.~properly stable vector). In this case, 
  the set $V^{s}=V^s(G)$  of $G$-stable (resp.~properly stable) vectors is Zariski open.
We denote by $V^{\sss}=V^{\sss}(G)$ the closed subset $V \setminus V^{s}$; vectors $v \in V \setminus V^{s}$ are said to be {\em $G$-strictly semi-stable}.
  \end{definition}
  
\begin{definition}
\label{def:V-is-pure}
Let $V$ be a stable representation of a connected reductive group $G$ and let $\pi\colon V\to V/G$ be the quotient map. Then $V$ is \emph{pure} if $\pi(V^{\sss})$ is pure of codimension-one in $V/G$.
\end{definition}

Despite the fact that $p$ is not proper quasi-finite and that $\pi(V^{\sss})$ is not a perfect analogue of the branch locus of
$p$, our crude GIT analogy is already remarkably powerful, as
witnessed by Theorem \ref{thm:simple} below. Recall that a
$G$-representation $V$ is called \emph{coregular} if $V/G$ is smooth,
and is called \emph{cofree} if it is coregular and $\pi\colon V\to V/G$ is flat.

\begin{theorem}\label{thm:simple}  
Let $V$ be an irreducible stable representation of a simple Lie group $G$. Then $V$ is cofree if and only if $V$ is pure.
\end{theorem}


For groups $G$ with non-trivial characters, a more refined notion of purity is needed. By Remark \ref{rmk:purity-notions-equivalent}, this more refined notion is equivalent to purity for simple Lie groups.

\begin{definition}
Let $V$ be a stable representation of a connected reductive group $G$ and let $\pi\colon V\to V/G$ be the quotient map. We say $V$ is \emph{coprincipal} if it is pure and every irreducible component of $V^{\sss}$ (with its reduced subscheme structure) maps to a
  principal divisor under $\pi$.
  \footnote{
  Note that any locally principal divisor on $V/G$ is principal because $\Pic(V/G) = 0$ \cite{KKV:89}}
  \end{definition}

\begin{remark}
  \label{rmk:purity-notions-equivalent}
  If $G$ has no non-trivial characters then it is easy to show (Lemma \ref{part:sl_n-reps}.\ref{l:allpure})
  that the notions of pure and coprincipal coincide, and that a sufficient
  condition for $V$ to be pure is that $\codim(V^{sss}) = 1$.
  In contrast, when $G$ is a torus, we exhibit pure representations which are not coprincipal; we also give representations where
  $V^{sss}$ is a divisor but $V$ is not pure, i.e.~$\pi(V^{sss})$ is not pure of codimension-one. See Example \ref{part:torus-actions}.\ref{ex:counter-example}.
\end{remark}


Further illustrating the utility of our GIT analogy, we prove:

\begin{theorem}
\label{thm:conj-for-tori}
A stable torus representation is cofree if and only if it is coprincipal.
\end{theorem}

\begin{remark}
\label{rmk:distinguishing-pure-pure-coprincipal-coorbifold}
Furthermore, we prove in Proposition \ref{part:torus-actions}.\ref{prop.qcnice} 
that if $V$ is a pure representation of a torus which is not coprincipal then $V/G$ has worse than finite quotient singularities.
  \end{remark}

\begin{remark}
  The restriction to stable representations in Theorem \ref{thm:conj-for-tori} is relatively
  insignificant because Wehlau \cite[Lemma 2]{Weh:92} proved that any
  torus representation $V$ has a (canonical) stable submodule $V'$
  such that $V'/T = V/T$ with the properties that $V' = V$ if and only if 
  $V$ is stable, and $V'$ is cofree if and only if $V$ is cofree. Thus
  Theorem \ref{thm:conj-for-tori} can be restated as saying that if
  $V$ is an arbitrary representation of a torus with non-trivial
  invariant ring, then $V$ is cofree if and only if the stable submodule $V'$
  is coprincipal.
  \end{remark}

Our main results relating coregularity, cofreeness, purity, and coprincipality for stable representations are summarized in the diagram below. Note that when $V$ is irreducible and $G$ is simple, all four notions coincide.
\[
\xymatrix{
 & \textrm{coregular}\ar@/^1.5pc/@{=>}[d]^-{\textrm{\cite{KPV:76}:~$V$\textrm{\ irreducible},\ $G$\textrm{\ simple}}} & \\
 & \textrm{cofree}
 \ar@/^1.5pc/@{=>}[u]^-{\textrm{by\ definition}} & \\
&&\\
\textrm{coprincipal}\ar@{<=>}[uur]^-{\textrm{Theorem\ } \ref{thm:conj-for-tori}:~\textrm{torus\ representations}}\ar@{=>}[rr]^-{\textrm{by\ definition}} & & \textrm{pure}
\ar@{<=>}[uul]_-{\textrm{Theorem\ } \ref{thm:simple}:~V\textrm{\ irreducible},\ G\textrm{\ simple}}
\ar@/^1.5pc/@{=>}[ll]^{\textrm{Remark \ref{rmk:purity-notions-equivalent}:~$G$ has no non-trivial characters}}
}
\]

\subsection{Questions and  examples}
The fact that coprincipality characterize cofreeness for
irreducible representations of simple Lie groups as well as reducible
representations of tori suggestions that it may be a useful class of
representations in greater generality. We pose the following
questions.

\begin{question}
\label{q:nicely-pure}
Let $V$ be a stable representation of a connected reductive group $G$.
\begin{enumerate}[label=(\arabic*)]
\item\label{conj:nicely-pure::cofree->pure} If $V$ is cofree, then is it coprincipal?

\item\label{conj:nicely-pure::main} Let $G$ be semisimple and $V$ be irreducible. If $V$ is pure (equivalently coprincipal), then is it cofree?
\end{enumerate}
\end{question}

\begin{remark}[Relationship to a result of Brion]
Michel Brion pointed us to a result of his \cite[4.3 Corollaire 1]{Bri:93} which gives some evidence for an affirmative answer to Question \ref{q:nicely-pure}\ref{conj:nicely-pure::cofree->pure}. Precisely, Brion proves that if $V$ is a properly stable representation of a reductive group (not necessarily connected), and if $\codim(V \smallsetminus V^{\pr}) \geq 2$, then $K[V]$ cannot be a free $K[V]^G$ module. Here $V^{\pr}$ is the locus of orbits of principal type. Since $V^{\pr} \subset V^{s}$, it follows that such representations are not pure.
\end{remark}

\begin{remark}[Reducible representations]
We note that the irreducibility assumption in Question \ref{q:nicely-pure} \ref{conj:nicely-pure::main} cannot be dropped, even when $G$ is simple. The smallest example we know to illustrate this, which we learned from Gerald Schwarz, is the $\SL_3$-representation $V = \Sym^2(\CC^3) \oplus (\CC^3)^{\oplus 2}$. The fact that the irreducibility assumption cannot be dropped is completely analogous to the picture for finite groups. Indeed, for a faithful representation $V$ of a finite group $G$, if $V^f$ denotes the open set on which $G$ acts freely, then the condition that $V\smallsetminus V^f$ is a divisor is necessary but not sufficient for cofreeness of $V$. For a simple example, consider the $\mu_4$-action on $\CC^2$ with weights $(1,2)$.

Although reducible, pure, non-cofree representations do exist, the
conditions of purity and cofreeness are both quite rare for reducible
representations. Indeed, for semisimple groups, any representation with
no trivial summands and at least two properly stable summands cannot be cofree; similarly, for a reductive group, any representation with at least two properly stable summands 
cannot be pure. To see the former statement, note that any properly stable $G$-representation has dimension at least $\dim G +1$, and by \cite[Theorem 8.9]{PoVi:94} if $G$ is semisimple then any coregular (and thus cofree) representation with no trivial summands has dimension at most $2 \dim G$. To see the latter statement, note that if $V$ and $W$ are properly stable
representations then the Hilbert--Mumford criterion implies that $(V^s
\oplus W) \cup (V \oplus W^s) \subset (V \oplus W)^{s}$. Since
$\codim_V (V\smallsetminus V^s) \geq 1$ and $\codim_W (W
\smallsetminus W^s) \geq 1$, we see that $\codim_{V\oplus W}((V \oplus
W)\smallsetminus (V \oplus W)^{s}) \geq 2$.
\end{remark}

\section{Outline of the proofs of the main theorems}
\label{subsec:outline-sln-pf}

\subsection{Theorem \ref{thm:simple}:~``only if'' direction}
This is the easier direction of Theorem \ref{thm:simple}. Recall that a representation of $V$ is {\em polar} if
there is a subspace $\fc \subset V$ and a finite group $W$ such that
$K[V]^G = K[\fc]^W$. The basic example of a polar representation is
the adjoint representation $\fg$; here $\fc$ is a Cartan subalgebra
and $W$ is the Weyl group. Using results of Dadok and Kac
\cite{DaKa:85} we prove that any stable polar representation (not necessarily irreducible) is pure, see Proposition \ref{part:sl_n-reps}.\ref{prop:polar-reps-Vsss-div}. On the other hand, Dadok and Kac proved that any irreducible cofree representation of a simple group is polar. Thus, we conclude that any stable irreducible cofree representation of a simple group is pure.

\subsection{Theorem \ref{thm:simple}:~``if'' direction}
This is the most involved result of the paper. In Section \ref{part:sl_n-reps}.\ref{sec:bounding-pure-reps}, we show that if $G$ is reductive and $V$ is an pure $G$-representation, then there is a hyperplane $H$ in the character lattice of $V$ tensored with $\RR$ satisfying the following special
condition: $H$ contains at least $\dim V - \dim G + 1$ weights when
counted with multiplicity. In particular, this implies that when
$G$ is semisimple, every irreducible pure representation $V$ has dimension bounded by a cubic in $\rk(G)$. 
In Section \ref{part:sl_n-reps}.\ref{sec:small-enough-representations}, we further
show that if $V$ is pure, then its highest weight lies on a ray
(or possibly a $2$-dimensional face, if $G = \SL_n$) of the Weyl chamber. Comparing with the
known list of cofree representations of simple groups, we are reduced to
checking that $11$ infinite families and $94$ more sporadic cases, are not pure.
These calculations, performed in Section \ref{part:sl_n-reps}.\ref{sec:notpure} are mostly done by computer, but a number must be done by hand, and will show the nature of the computer analysis done.

\subsection{Theorem \ref{thm:conj-for-tori}}
We prove Theorem \ref{thm:conj-for-tori} for tori $T$ by inducting on $\dim V$. The key to the proof is showing in Proposition
\ref{part:torus-actions}.\ref{prop:1D-factor} that if $V$ is a coprincipal representation
of a torus, then $V$ splits as a sum of $T$-representations $V= V_1 \oplus V_2$ such that $V/T = V_1/T \times V_2/T$
and $V_1/T$ is 1-dimensional. This argument makes essential use of the fact
that the images of the irreducible components of $V^{\sss}$ are Cartier divisors. 


\section*{Acknowledgments}
It is a pleasure to thank Jason Bell, Chris Mannon, Matt Kennedy, Radu Laza, Heydar Radjavi, Frank Sottile, and Jerry Wang for several helpful conversations. We are especially grateful to William Slofstra,  Ronan Terpereau, and Ben Webster for many illuminating conversations.

We thank Michael Bulois for pointing out an error in the first version of the paper. We are grateful to Michel Brion and Victor Kac for their insightful comments and interest in our work.

Finally, we are beyond grateful both to an anonymous referee whose suggestions greatly improved our paper, and to Gerald Schwarz whose wealth of
examples helped shape Question \ref{q:nicely-pure}. We thank him for allowing us to include his examples in this article.

\part{Representations of simple groups: proof of Theorem \ref{part:general-facts}.\ref{thm:simple}}
\label{part:sl_n-reps}

This part is organized as follows.
In \S\ref{sec.basicfacts} we prove some basic facts about pure representations that we will use throughout.
In \S\ref{sec:polar}, we prove that every stable cofree irreducible representation of a connected simple group $G$ is pure, that is, we prove the ``only if'' direction of Theorem \ref{part:general-facts}.\ref{thm:simple}.
In \S\ref{sec:bounding-pure-reps}, we prove the key result that if $V$ is any pure representation, then there is a hyperplane in the weight lattice containing most of the weights.
This implies that up to isomorphism there are a finite number of pure representations not containing a trivial summand.

In \S\ref{sec:small-enough-representations}, we apply the criteria in \S\ref{sec:bounding-pure-reps} to create a finite list on which all irreducible pure representations of a simple group may be found.
Then, in \S\ref{sec:notpure}, we demonstrate that representations on this list which are not cofree are also not pure, thus proving Theorem \ref{part:general-facts}.\ref{thm:simple}.

\section{Basic facts about pure representations} \label{sec.basicfacts}

\begin{lemma}
\label{l:pureextension}
If $V$ is a representation of a reductive group $H$, and $f : G \to H$ is a surjection, then $V$ is pure when considered as an $H$-representation if and only if $V$ is pure when considered as a $G$-representation.
\end{lemma}
\begin{proof}
Since $f$ is surjective, $Hv = f(G)v$, which is $Gv$ by definition of the $G$-action.
So all the orbits of the actions are the same, and hence their purity is the same.
\end{proof}

In the event that $G$ is a simple Lie group which is not simply connected, it has a universal covering $G' \to G$ with finite kernel $K$.
There is a one-to-one correspondence between $G$-representations and $G'$-representations which are trivial when viewed as a $K$-representations. By Lemma \ref{l:pureextension}, we have
\begin{corollary}
\label{c:simplyconnected}
To prove Theorem \ref{thm:simple} for all simple Lie groups, it suffices to prove it for the exceptional groups and the groups $\SL_n, \Sp_{2n}, \Spin_n$ for all $n$.
\end{corollary}

\begin{lemma} \label{l:allpure}
Let be $G$ a reductive group with no non-trivial characters. If $V$ is stable representation for which $V^{\sss}$ is pure of codimension-one, then $V$ is coprincipal. In particular, every pure representation of $G$ is coprincipal.
\end{lemma}
\begin{proof}
Since $G$ is connected, every component of $V^{\sss}$ is $G$-invariant.
Thus, the equation $f \in K[V]$ of the component must be an eigenfunction for the action of $G$ on $K[V]$; that is, $g\cdot f = \lambda(g) f$ for all $g \in G$.
Since $G$ has no non-trivial characters, $f$ must in fact be invariant.
Hence the image of $V(f)$ is the Cartier divisor defined by $f$ in $\spec K[V]^G$.
\end{proof}

\begin{lemma}
\label{l:outerauto}
If $\rho, \rho'$ are representations of $G$ on a vector space $V$, and there exists an automorphism $f$ of $G$ exchanging $\rho$ and $\rho'$, then $\rho$ is pure (resp.~cofree) if and only if $\rho'$ is pure (resp.~cofree).
\end{lemma}
\begin{proof}
Purity and cofreeness are both determined at the level of the image $\rho'(G) = \rho(f(G)) = \rho(G) \subseteq \GL(V)$.
\end{proof}


\begin{remark}
\label{r:outerauto}
Lemma \ref{l:outerauto} will be used for the spin groups in the following fashion: for each spin group $\Spin_{2n}$ of even order, there exists an outer automorphism exchanging the half-spinor representation $\Gamma_{\omega_{2n}}$ and the half-spinor representation $\Gamma_{\omega_{2n-1}}$, and it follows that the positive half-spinor representation is cofree and pure if and only if the negative half-spinor representation is cofree and pure.

For $\Spin_8$ there is a triality which gives automorphisms exchanging the half-spinor representations and the standard representation.
Since the standard representation is cofree and pure, both half-spinor representations of $\Spin_8$ are cofree and pure.
Moreover, these outer automorphisms send the symmetric square $\Sym^2 \CC^8$ with highest weight $2\omega_1$ to the representations $\Gamma_{2\omega_7}, \Gamma_{2\omega_8}$, which are the irreducible components of the wedge product $\extp^4 \CC^8$.
Since $\Sym^2 \CC^8$ is cofree and pure as a $\Spin_8$ representation, so too are both irreducible components of $\extp^4 \CC^8$.
\end{remark}

For low dimensional special orthogonal Lie algebras, there exist exceptional isomorphisms $\sp(4) = \so(5)$ and $\sl(4) = \so(6)$.
These obviously preserve purity and cofreeness of representations, and so it suffices for the classical groups to prove Theorem \ref{thm:simple} for $\SL_n$ when $n \ge 2$, $\Spin_{2n+1}$ when $n \ge 2$, $\Sp_{2n}$ when $n \ge 3$, and $\Spin_{2n}$ when $n \ge 4$.

\section{Proof of the ``only if'' direction of Theorem \ref{part:general-facts}.\ref{thm:simple}}
\label{sec:polar}
The proof of the ``only if'' direction of Theorem \ref{part:general-facts}.\ref{thm:simple} is relatively straightforward thanks to the work of Dadok and Kac on polar representations \cite{DaKa:85}.
Recall that a representation $V$ of a reductive group $G$ is \emph{polar} if there exists a subspace $\cC$, called a \emph{Cartan subspace}, such that the map $\cC \to \spec K[V^G]$ is finite and surjective.

In \cite[Theorem 2.9]{DaKa:85}, Dadok and Kac proved that if $V$ is polar with Cartan subspace $\cC$, then the group $W = N_G(\cC)/Z_G(\cC)$ is finite and $K[V]^G = K[\cC]^W$.
By \cite[Theorem 2.10]{DaKa:85}, every polar representation is cofree. Furthermore, using the classification of irreducible cofree representations of simple groups, they showed that every irreducible cofree representation of a simple group is polar.

As a result, to prove the ``only if'' direction of Theorem \ref{part:general-facts}.\ref{thm:simple}, it is enough to show that polar representations are pure. We are grateful to Ronan Terpereau for suggesting this proof.

\begin{proposition}
\label{prop:polar-reps-Vsss-div}
If $V$ is a stable polar representation (not necessarily irreducible), then it is pure.
\end{proposition}
\begin{proof}
Let $\cC$ be a Cartan subspace, and following \cite[p.~506]{DaKa:85}, let $\cC^{reg}$ be the set of regular points.
By definition, $v \in \cC^{reg}$ if and only if $Gv$ is closed and of maximal dimension among closed orbits.
If $V$ is a stable representation then this condition is equivalent to the stability of $v$.
Since the Cartan subspace contains a point of each closed $G$ orbit, it follows that $G\cC^{reg} = V^{s}$.

By \cite[Lemma 2.11]{DaKa:85}, $\cC_{sing} = \cC \setminus \cC^{reg}$ is a finite union of hyperplanes, and $V^{\sss} = G \cC_{sing}$ by definition.
If $W$ is as above, then the image of $\cC_{sing}$ under the quotient map $p \colon \cC \to \cC /W$ is a divisor.
The composition $\cC \hookrightarrow V \stackrel{\pi} \to V/G$ is finite by \cite[Proposition 2.2]{DaKa:85}.
Under the identification $V/G = \cC /W$ this finite map is just the quotient map $\cC \to \cC / W$.
Thus every irreducible component of $V^{\sss} = \pi^{-1}(p(\cC_{sing}))$ is a divisor, because $V \to V/G$ is flat, as $V$ is cofree by \cite[Theorem 2.10]{DaKa:85}. By By Lemma \ref{l:allpure}, we see $V$ is pure.
\end{proof}
\begin{remark}
As noted by Victor Kac, there are polar representations with non-trivial rings of invariants where our proposition does not apply.
However, an analogous statement holds with $V^s$ replaced by the $G$-saturation of the locus of closed orbits, which are of maximal dimension among closed orbits.
\end{remark}

\section{Bounding the dimension of a pure representation} \label{sec:bounding-pure-reps}
We begin by obtaining results that show pure representations are relatively rare; specifically, any simple group has a finite number of pure representations that do not contain a trivial summand.

The following result holds for any reductive group, not just simple or semi-simple ones.

\begin{proposition} \label{prop:hyperplanes}
Let $V$ be a stable representation of a reductive group $G$.
Suppose $V^{\sss}$ contains a divisorial component that maps to a divisor in $V\sslash G$, e.g.~$V$ is pure.
Then there exists one-parameter subgroup $\lambda$ such that
\[ \dim V_\lambda^{0} \geq \dim V - \dim G -1, \]
where $V_\lambda^{0}\subset V$ denotes the 0-weight space of $\lambda$; said differently, there are at most $\dim G+1$ weights that do not lie on the hyperplane of the weight space defined by $\lambda$.
\end{proposition}

\begin{proof}
First assume that $V$ is properly stable. In this case, every stable vector has finite dimensional stabilizer.
Hence, a vector $v$ is not stable if and only if it contains a point with positive dimensional stabilizer in its orbit closure.
Any closed orbit in a representation is affine so its stabilizer is reductive by \cite{Mat:60}, so if it is positive dimensional then it contains a one-parameter subgroup.
Thus $v \in V^{\sss}$ if and only if there is a 1-parameter subgroup $\lambda$ such that $v \in V_{\lambda}^{\geq 0}$, where $V_{\lambda}^{\geq 0}$ is the subspace of $V$ whose vectors have non-negative weight with respect to $\lambda$.

Since all one-parameter subgroups are $G$-conjugate, it follows that $V^{\sss} = \cup_{\lambda \in N(T)} G V_{\lambda}^{\geq 0}$ where $N(T)$ is the group of one-parameter subgroups of a fixed maximal torus $T$.
Since $V$ is finite dimensional, it contains a finite number of weights, so there are only finitely many distinct subspaces $V_{\lambda}^{\geq 0}$ as $\lambda$ runs through the elements of $N(T)$.
Hence, there exists a one-parameter subgroup $\lambda$ such that $GV_{\lambda}^{\geq 0}$ is the divisorial component of $V^{\sss}$.
Since $GV_{\lambda}^{\geq 0}$ is the $G$-saturation of the fixed locus $V_\lambda^0$, it follows that $\pi(GV_{\lambda}^{\geq 0}) = \pi(V_\lambda^0)$, where $\pi \colon V \to V\sslash G$ is the quotient map.
Hence $\pi(V_{\lambda}^0)$ is a divisor in $V\sslash G$, and it therefore has dimension $\dim V - \dim G - 1$.
Thus $\dim V_{\lambda}^0 \geq \dim V - \dim G -1$.
 
When $V$ is stable but not properly stable, it is still the case that any strictly semi-stable point contains a point with positive dimensional stabilizer in its orbit closure.
The same argument used above implies that $V^{\sss} \subset \bigcup_{\lambda \in N(T)} G V_{\lambda}^{\geq 0}$.
It follows that any divisorial component of $V^{\sss}$ is contained in $GV_{\lambda}^{\geq 0}$ for some $1$-parameter subgroup $\lambda$.
If this divisorial component maps to a divisor, then image of $V^0_\lambda$ contains a divisor, so we conclude that $\dim V^0_{\lambda} \geq \dim V\sslash G -1 \geq \dim V - \dim G -1$.
\end{proof}

\begin{remark}
\label{rem:duals}
It is clear that if a representation can be shown to not be pure by Proposition \ref{prop:hyperplanes}, then its dual representation will also not be pure, since it will have the same number of weights in a hyperplane.
\end{remark}

\begin{example}
In this example, we illustrate that if $V$ is not properly stable, 
then $V^{\sss}$ may be a proper subset of
$\bigcup_{\lambda \in N(T)} G V_{\lambda}^{\geq 0}$. Let $V$ be the adjoint representation of $\SL_2$.
  Then the strictly semi-stable locus $V^{\sss}$ is the divisor defined by the vanishing of the determinant.
  However, since the torus of $\SL_2$ is rank one, the fixed locus $V_{\lambda}^0$ is
  the same for all $\lambda$ and is one dimensional. In this case,
  $V = \bigcup_{\lambda \in N(T)} G V_{\lambda}^{\geq 0}$
  and
  $\pi(V_\lambda^0)
  = \pi(V) = \AA^1$.
  \end{example}

We are grateful to the referee for pointing out the following consequence of Proposition \ref{prop:hyperplanes}, which has been used to significantly simplify the computations in the paper:
\begin{lemma}
\label{l:toral}
Suppose $\{t_1, \ldots, t_k\}$ are a complete set of Weyl group orbit representatives of order $\ell$ toral elements\footnote{That is, elements contained in some fixed choice of maximal torus of $G$ that act on $V$ with eigenvalues of the form $\exp(2\pi \iota k/\ell)$ for some $k$.} in the maximal torus of a simple Lie group $G$.
If $\overline{\pi(V^{t_i})} \subseteq V\sslash G$ has codimension at least 2 for all $i$, then $\codim V^{\sss} \ge 2$.
Moreover,
\[
\dim \overline{\pi(V^{t_i})} \le \dim V^{t_i} \sslash N_{t_i} \le \dim V^{t_i},
\]
where $N_{t_i}$ denotes the normalizer of $t_i$.
\end{lemma}
\begin{proof}
Any 1-parameter subgroup $\lambda$ contains some toral element $t$ of order $\ell$.
Thus an upper bound for $\dim \pi (V^{\lambda})$ is $\max\{ \dim \pi(V^t) \mid t \text{ of order $\ell$} \}$.
In particular, if each $\overline{\pi(V^t)}$ has codimension at least 2, then $\pi (V^{\lambda})$ has codimension at least 2, and by Proposition \ref{prop:hyperplanes}, $\dim V^{\sss} \ge 2$.
Moreover, the morphism $V^t \to \pi(V^t)$ factors through $V^t \sslash N_t$ of $t$, whence $\dim \overline{\pi(V^t)} \le \dim V^{t_i} \sslash N_{t_i} \le \dim V^{t_i}$.
\end{proof}

In practice, Lemma \ref{l:toral} is much easier to use than Proposition \ref{prop:hyperplanes}.
Software packages like LiE \cite{LiE:92} are capable of computing $\dim V^t$ given a representation of a semisimple Lie group and a toral element $t$ of finite order, and so it is possible to show many representations are not pure by checking that $\dim V^t \le \dim V - \dim G - 2$ for all $t$ up to the action of the Weyl group.
For representations 
where this bound is not \emph{a priori} clear
and for representations occurring in infinite families, a more careful analysis can be done by hand.
In some exceptional cases, it is still easier to count the number of weights occurring in a hyperplane.

An important use of Proposition \ref{prop:hyperplanes} is to show that a simple Lie group has finitely many pure representations. This can be shown by establishing a bound on the dimension of a pure representation which is polynomial in the rank of the group.
To show this bound, we separately bound with multiplicity the number of zero and non-zero weights that can occur in the representation.
\begin{proposition}
\label{prop:bound-non0-wts}
Let $V$ be a (stable) representation of a simple Lie group $G$, and suppose that $V^{\sss}$ contains a divisorial component.
Then there are at most $\rk(G)(\dim(G) + 1)$ non-zero weights counted with multiplicity.
\end{proposition}
\begin{proof}
Since $G$ is simple, the image of a divisorial component of $V^{\sss}$ in $V\sslash G$ is also a divisor by Lemma \ref{part:general-facts}.\ref{l:allpure}.
Hence by Proposition \ref{prop:hyperplanes} there is a 1-parameter subgroup $\lambda$ such that $V_\lambda^0$ contains at least $\dim V - \dim(G) - 1$ weights.
Let $H$ be the hyperplane in the character lattice determined by this one-parameter subgroup, so Proposition \ref{prop:hyperplanes} says that $H$ contains at least $\dim V - (\dim(G)+1)$ such weights.
The Weyl group conjugates of $H$ also contain at least $\dim V - (\dim(G)+1)$ weights.
We claim that $H$ has at least $\rk(G)$ linearly independent conjugates under the Weyl group.
Suppose $v$ is a normal vector to $H$.
If $W$ is the Weyl group, then it acts linearly on $\fh^*$, and the representation so obtained is irreducible by \cite[Lemma 14.31]{FuHa:91}. Since the linear span of the orbit $Wv$ is a non-zero $W$-invariant subspace of $\fh^*$, it must be all of $\fh^*$, and so the number of linearly independent conjugates of $v$ under the Weyl group is $\dim \fh^* = \dim \fh = \rk(G)$.

Let $H = H_1, \ldots, H_{\rk(G)}$ be $\rk(G)$ such conjugate hyperplanes whose normal vectors are linearly independent.
By inclusion-exclusion, $H_1 \cap H_2$ contains at least $\dim V - 2(\dim(G) + 1)$ weights counted with multiplicity, since $H_1 \cup H_2$ contains at most $\dim V$ weights.
Assume by induction that $H_1 \cap \cdots \cap H_k$ contains at least $\dim V - k(\dim(G) + 1)$ weights counted with multiplicity.
Since $(H_1 \cap \cdots \cap H_k) \cup H_{k+1}$ still contains at most $\dim V$ weights, the inclusion-exclusion principle implies that $(H_1 \cap \cdots \cap H_k) \cap H_{k+1}$ contains at least $\dim V - (k+1)(\dim(G)+1)$ weights counted with multiplicity.
Hence, $\{0\} = H_1 \cap \cdots \cap H_{\rk(G)}$ contains at least $\dim V - \rk(G)(\dim(G)+1)$ weights counted with multiplicity.
In other words, the multiplicity of the 0 weight in $V$ is at least $\dim V - \rk(G)(\dim(G)+1)$.
\end{proof}

\begin{lemma}
\label{l:0wt-injects-simple-pos-roots}
Let $V$ be an irreducible representation of a semi-simple Lie group $G$, and let $\alpha_i$ be the positive simple roots. If $a$ is any non-highest weight for $V$, then
\[ \dim(V_a) \leq \sum_i \dim(V_{a+\alpha_i}). \]
In particular, if $V$ is not the trivial representation, then $\dim(V_0) \leq \sum_i \dim(V_{\alpha_i})$.
\end{lemma}
\begin{proof}
Consider the linear map $V_a \to \bigoplus_i V_{a+\alpha_i}$ given by $v\mapsto (e_i(v))$ where $e_i$ is the root vector in the Lie algebra for $\alpha_i$.
Since $a$ is not a highest weight, $V_a$ does not contain a highest weight vector, that is, no vector $v\in V_a$ is killed by all positive simple roots.
Hence, the above map is injective.
\end{proof}

\begin{lemma}
\label{l:zero-wt-ratio-bd}
Let $V$ be a representation of a simple Lie group $G$ which contains no trivial summands.
Then
\[ \frac{\dim(V_0)}{\dim(V)} \le \frac{\rk(G)}{\dim(G) - \rk(G)}. \]
\end{lemma}
\begin{proof}
It clearly suffices to prove the lemma for every non-trivial irreducible subrepresentation of $V$, and so we may assume $V$ is irreducible.
Let $d = \dim V$, let $d_0 = \dim V_0$ be the dimension of the 0-weight space, and let $d_\alpha = \dim V_\alpha$ be the dimension of the weight space of any root $\alpha$.
There are $\dim(G) - \rk(G)$ total roots, so we obtain the inequality $(\dim(G) - \rk(G))d_\alpha + d_0 \le d$. Of the $\dim(G) - \rk(G)$ roots, $\rk(G)$ of them are simple, so by Lemma \ref{l:0wt-injects-simple-pos-roots}, $d_0 \le \rk(G)d_\alpha$.
Thus,
\[ \frac{d_0}{d} \le \frac{\rk(G)d_\alpha}{(\dim(G)-\rk(G))d_\alpha + d_0} \le \frac{\rk(G)}{\dim(G)-\rk(G)}. \]
\end{proof}

\begin{proposition}
\label{prop:bounddimV}
Let $V$ be a (stable) representation of a simple Lie group $G$ which contains no trivial summands and such that $V^{\sss}$ contains a divisor.
If $0$ is not a weight of $V$, then $\dim V \le \rk(G)(\dim(G)+1)$.
If $0$ is a weight of $V$, then $\dim V \le \frac{\rk(G)(\dim(G)+1)(\dim(G)-\rk(G))}{\dim(G)-2\rk(G)}$.
\end{proposition}
\begin{proof}
Proposition \ref{prop:bound-non0-wts} proves the claim when $0$ is not a weight.
Assume $0$ is a weight.
Since $V^{\sss}$ is a divisor, it follows from Proposition \ref{prop:bound-non0-wts} that $d \le \rk(G)(\dim(G)+1) + d_0$ where $d = \dim V, d_0 = \dim V_0$.
Then
\[ d \le \rk(G)(\dim(G)+1) + d_0 \le \rk(G)(\dim(G)+1) + \frac{\rk(G)}{\dim(G)-\rk(G)} d \]
by Lemma \ref{l:zero-wt-ratio-bd}, which implies that
\[ d \le \frac{\rk(G)(\dim(G)+1)(\dim(G)-\rk(G))}{\dim(G)-2\rk(G)}.\qedhere \]
\end{proof}

For the classical groups, we make use of slightly weaker bounds than those in Proposition \ref{prop:bounddimV}.

\begin{definition}
\label{def:smallenough}
Let $G$ be a classical group and $V$ an irreducible $G$-representation. Then $V$ is \emph{small enough} if $\dim V\leq\kappa(G)$, where
\[
\kappa(G)=\begin{cases}
n^3, & G=\SL_n \\
2(n^3+n^2+n+1)+1, & G=\Sp_{2n}\\
2(n^3+n^2+n+1) +1, & G=\SO_{2n+1}\\
139, & G=\SO_{8}\\
2(n^3+n+3/2)+1, & G=\SO_{2n}, n > 4.
\end{cases}
\]
In this case, we often say the highest weight of $V$ is small enough.
\end{definition}

\begin{remark}
\label{r:smallenough}
By Proposition \ref{prop:bounddimV}, a pure irreducible representation of a simple Lie group is small enough.
The problem of enumerating the pure representations may thus be reduced to first enumerating the small enough representations, and then determining which of them are pure.
\end{remark}

\begin{remark}
\label{rmk:smallenough-ss}
The bound of Proposition \ref{prop:bounddimV} holds for representations of semisimple Lie groups also.
Unlike the case of simple Lie groups, for semisimple Lie groups difficulties arise in trying to create lists of representations meeting these bounds primarily because if $G_1$ is a simple Lie group with large rank and $G_2$ is a simple Lie group with much smaller rank, then $V_1\otimes V_2$ frequently gives a representation which is small enough but is not pure or cofree.
The sizes of the lists in this case are infeasible to handle with the methods presented here.
\end{remark}

\section{Small enough representations}
\label{sec:small-enough-representations}

Let $\fg$ be a simple Lie algebra with Cartan algebra $\fh$ and fundamental weights $\omega_1, \ldots, \omega_n \in \fh^*$.
Denote by $\Gamma_{a_1\omega_1 + \cdots a_n\omega_n}$ the representation of $\fg$ with highest weight vector $a_1 \omega_1 + \ldots + a_n \omega_n$.
Define the \emph{width} of the vector $a_1 \omega_1 + \ldots + a_n \omega_n$ to be $a_1 + \cdots + a_n$, and define its \emph{support} to be the number of $i$ such that $a_i \neq 0$.
If $\omega = a_1 \omega_1 + \ldots + a_n \omega_n$ is a weight, and $i, j$ are such that $a_i, a_j \neq 0$, define the \emph{shift} $\omega_{i \to j}$ to be $\omega + a_i (\omega_j - \omega_i)$.
Note that $\omega$ and $\omega_{i \to j}$ both have the same width, and the support of $\omega_{i \to j}$ is one less than the support of $\omega$.
We use the following lemma, which is proven along the way to \cite[Lemma 2.1]{GoGuSt:17}.
\begin{lemma}
\label{l:shiftdim}
Suppose $(a_1, \ldots, a_n) \in \NN^n$, and let $\omega = a_1 \omega_1 + \ldots + a_n \omega_n$.
If $i, j$ are distinct indices such that $a_i, a_j \neq 0$, then
\[ \dim \Gamma_{\omega} \ge \min\left(\dim \Gamma_{\omega_{i \to j}}, \dim \Gamma_{\omega_{j \to i}} \right). \]\nobreak\hfill\ensuremath{\square}
\end{lemma}
In particular, if $\omega$ is small enough, then one of $\omega_{i \to j}, \omega_{j \to i}$ must be small enough.
This suggests the following algorithm to classify the small enough representations:
\begin{enumerate}
\item Find a closed form for the dimension of the irreducible representation with highest weight $\omega_s$ for each $s$, and classify the 1-supported small enough weights using these closed forms.
\item If $\omega$ is a small enough weight with support $\{i,j\}$, use Lemma \ref{l:shiftdim} to restrict the possibilities for $\omega$ and classify the 2-supported small enough weights.
\item If the list of 2-supported small enough weights is not empty, repeat the above proceduce for 3-supported small enough weights, and so forth, until the list of $K$-supported small enough weights is empty for some $K$.
\item If $\omega$ is a weight with support greater than $K$, by repeatedly applying Lemma \ref{l:shiftdim}, its dimension can be bounded below by that of a $K$-supported weight.
Since no $K$-supported weight is small enough, it follows that $\omega$ is not small enough for any support greater than $K$, and the list is complete.
\end{enumerate}

\subsection*{Reading the tables}
Each of the families of Lie groups in this section are accompanied by a list of all their small enough representations, see Tables \ref{t:sl}--\ref{t:spineven}. 
The representations are described by their highest weights, in the notation of \cite{FuHa:91}.

The final column in each table describes whether the indicated representation is cofree and/or pure for the indicated values of $n$.
By the ``only if'' direction of Theorem \ref{part:general-facts}.\ref{thm:simple}, cofree implies pure for a stable irreducible representation of a simple group, so ``cofree'' is written to mean ``cofree, pure, and stable''.
The representations meeting this criterion may be looked up in a list such as \cite{KPV:76}.
Note that these lists contain representations up to outer automorphism; as a result, entries on the tables in this section occasionally do not appear on them.
Such representations are be indicated as cofree with a reference to the relevant tool needed to find it in the list of \cite{KPV:76}.

An entry of ``impure'' similarly means ``stable, impure, and not cofree.''
When it is not otherwise specified, a representation marked ``Impure'' was proven to not be pure by a computer analysis using the software LiE \cite{LiE:92} and Lemma \ref{l:toral}.
For this paper, toral elements of orders 2 and 3 were used in this fashion, and LiE was used to compute the dimensions $\dim V^t$ for all of these $t$.
In the event that $\dim V^t$ is too large, or for infinite families of representations, a reference is be given to a statement in Section \ref{sec:notpure} in which the impurity of the representation is shown.

Representations that are not stable further have ``unstable'' in the final column.
The fact that each small enough representation is either unstable, cofree, or impure, proves Theorem \ref{thm:simple} for simple simply connected Lie groups, and then for all simple Lie groups by Corollary \ref{c:simplyconnected}.

\subsection{The case of $\SL_n$}
A representation $V$ of $\SL_n$ is small enough if $\dim V \le \kappa(\SL_n) = n^3$.
Note that the dual representation $\Gamma_{\sum a_i \omega_i}^*$ is $\Gamma_{a_{n-1} \omega_1 + \cdots + a_1 \omega_{n-1}}$.
It therefore suffices to classify the representations with at least half their support in $a_1, \ldots, a_{\lfloor n/2\rfloor}$, and the others are the dual representations.
By Remark \ref{rem:duals}, duality does not change purity or cofreeness as long as our proofs use Proposition \ref{prop:hyperplanes} to show impurity, and hence these representations shall be be ignored.
For example, in the case of 1-supported weights, we only consider highest weights $t\omega_s$ where $n \ge 2s-1$.

For $\SL_2$, every irreducible representation is of the form $\Sym^k V$ for some $k$, where $V$ is the defining representation.
The maximum number of weights contained in a hyperplane is then $0$ if $k$ is odd, or $1$ if $k$ is even.
So if $k > 4$, then $\Sym^k V$ is not pure by Proposition \ref{prop:hyperplanes}; one can check that when $k \le 4$, the representation is both cofree and pure.
In the sequel we consider representations of $\SL_n$ for $n > 2$.

\begin{table}[]
\begin{tabular}{|l|l|l|l|}
\hline 
1 & Highest weight & Restrictions on $n$ & Purity/Cofreeness \\ \hline
2 & $k\omega_1$ & $n = 2, 1 \le k \le 4$ & Cofree \\
3 & $k\omega_1$ & $n = 2, 5 \le k \le 8$ & Impure \\
4 & $\omega_1$ & $n \ge 3$ & Unstable \\
5 & $2\omega_1$ & $n \ge 3$ & Cofree \\
6 & $3\omega_1$ & $n \le 3$ & Cofree \\
7 & $3\omega_1$ & $n \ge 4$ & Impure, Lemma \ref{l:sln-extra} \\
8 & $4\omega_1$ & $3 \le n \le 17$ & Impure  \\
9 & $5\omega_1$ & $3 \le n \le 4$ & Impure \\
10 & $\omega_2$ & $n \ge 4$ & Cofree; Unstable for odd $n$ \\
11 & $2\omega_2$ & $ n = 4$ & Cofree \\
12 & $2\omega_2$ & $5 \le n \le 12$ & Impure \\
13 & $3\omega_2$ & $n = 4$ & Impure \\
14 & $\omega_3$ & $6 \le n \le 9$ & Cofree \\
15 & $\omega_3$ & $n \ge 10$ & Impure, Lemma \ref{l:sl-omega3}  \\
16 & $2\omega_3$ & $n = 6$ & Impure \\
17 & $\omega_4$ & $n = 8$ & Cofree \\
18 & $\omega_4$ & $9 \le n \le 29$ & Impure \\
19 & $\omega_5$ & $10 \le n \le 15$ & Impure \\
20 & $\omega_6$ & $12 \le n \le 13$ & Impure \\
21 & $\omega_1 + \omega_{n-1}$ & $n \ge 3$ & Cofree (Adjoint) \\
22 & $2\omega_1 + \omega_{n-1}$ & $n \ge 3$ & Impure, Lemma \ref{l:sln-extra} \\
23 & $\omega_1 + \omega_2$ & $n \ge 3$ & Impure, Lemma \ref{l:sln-extra} \\
24 & $\omega_1 + \omega_{n-2}$ & $n \ge 5$ & Impure, Lemma \ref{l:sln-extra} \\
25 & $\omega_1 + \omega_4$ & $n = 8$ & Impure \\
26 & $\omega_1 + \omega_3$ & $6 \le n \le 10$ & Impure \\
27 & $\omega_1 + \omega_{n-3}$ & $7 \le n \le 9$ & Impure \\
28 & $\omega_2 + \omega_3$ & $5 \le n \le 6$ & Impure \\
29 & $\omega_2 + \omega_4$ & $n = 6$ & Impure  \\
30 & $\omega_1 + \omega_2 + \omega_3$ & $n = 4$ & Impure \\
\hline
\end{tabular}
\caption{Small enough representations of $\SL_n$}
\label{t:sl}
\end{table}

\begin{lemma}
\label{l:support-1-sl}
The small enough irreducible representations of $\SL_n$ with 1-supported highest weights are the entries (1)--(20) of Table \ref{t:sl}.
\end{lemma}
\begin{proof}
Recall that $\Gamma_{\omega_s} = \extp^s \SL_n$ and so has dimension $\binom{n}{s}$.
When $4 \le s \le \left\lfloor \frac{n}{2} \right\rfloor$, it is clear that this dimension is asymptotically larger than $n^3$, and so there can be at most finitely many small enough such $\omega_s$.
We verify that the following are the only possible representations:
\[ \omega_4 (8 \le n \le 29), \omega_5 (10 \le n \le 15), \omega_6 (12 \le n \le 13). \]
We may also verify that no non-trivial scalar multiples of these representations are small enough.

Since $\Gamma_{t\omega_1} = \Sym^t \SL_n$, it has dimension $\binom{n+t-1}{t}$, and we see again that for $t \ge 4$ there can be at most finitely many solutions.
We can verify that these solutions are exactly
\[ 4\omega_1 (3 \le n \le 17), 5\omega_1 (3 \le n \le 4), 6\omega_1 (n=2), 7\omega_1 (n=2). \]
By \cite[Lemma 2.2]{GoGuSt:17}, $\dim \Gamma_{t\omega_s} \ge \dim \Gamma_{t\omega_1}$ for all $t, s$; using this we check that $t\omega_2$ and $t\omega_3$ are never small enough for $t \ge 4$.


For $t = 2,3$, we have $\dim \Gamma_{t\omega_1} \le n^3$ for all $n$, and so these representations are both small enough.
It remains only to determine the $n$ for which $2\omega_2$, $2\omega_3$, $3\omega_2$ and $3\omega_3$ are small enough.
The dimension of $2\omega_2$ is $\frac{n^2(n^2-1)}{12}$, which is at most $n^3$ when $4 \le n \le 12$, and we check for $n$ in this range that $3\omega_2$ is only small enough when $n = 4$.
The weight $2\omega_3$ is only small enough when $n = 6$, and $3\omega_3$ is not small enough for this $n$.
This completes the list of 1-supported small enough weights.
\end{proof}

\begin{lemma}
\label{l:support-2-sl}
The small enough 2-supported highest weights of $\SL_n$ are the entries (21)--(29) of Table \ref{t:sl}.
\end{lemma}
\begin{proof}
By Lemma \ref{l:shiftdim}, if $\omega = \omega_i + \omega_j$ is small enough, then one of $\omega_{i \to j} = 2\omega_j$ and $\omega_{j \to i} = 2\omega_i$ must be small enough.
Therefore, assuming without loss of generality that $i \le \lfloor n/2\rfloor$, we have that $i$ is one of $1, 2, 3$, with the latter two cases only possible if $n \le 12$ and $n = 6$ respectively.
If $i < j$, then we check that $i = 3$ is never small enough, and for $i = 2$ we have the possible cases $\omega_2 + \omega_3$ for $n = 5$, and $\omega_2 + \omega_3, \omega_2 + \omega_4$ for $n = 6$.

Otherwise, $i = 1$.
From \cite[Proposition 15.25(i)]{FuHa:91}, we find that $\Gamma_{\omega_1 + \omega_j} \otimes \extp^{j+1} \SL_n \simeq \extp^j \SL_n \otimes \SL_n$.
Therefore,
\[ \dim \Gamma_{\omega_1 + \omega_j} = n\binom{n}{j} - \binom{n}{j+1}. \]
If $\omega_1 + \omega_j$ is small enough, it must also be the case that $\omega_j$ is small enough.
Therefore, $j \in \{2,3,n-3,n-2,n-1\}$ unless $n$ is in some of the ranges $[8,29], [10,15], [12,13]$, where $j$ could also be in $\{4,n-4\}, \{5,n-5\}$, and $\{6,n-6\}$ respectively.
In the latter case, $\omega_1 + \omega_4$ is small enough for $n = 8$, no $\omega_1 + \omega_5$ are small enough for $n \in [10,15]$, and no $\omega_1 + \omega_6$ are small enough for $n = 12, 13$.
When $j \in \{2,n-2,n-1\}$, this formula shows that $\omega_1 + \omega_j$ is small enough.
For $j = 3$ we use the dimension formula to find that $\omega_1 + \omega_3$ is small enough for $n \le 10$.
Finally, we use the formula again to find that $\omega_1 + \omega_{n-3}$ is small enough for $n \le 9$.

Other than the weights $\omega_1 + \omega_2, \omega_1 + \omega_{n-2}$, and $\omega_1 + \omega_{n-1}$, the two-supported weights of weight two form a finite list, and we may check that increasing the width will not give a small enough weight; for example, $\omega_2 + \omega_3$ for $n = 5$ is small enough, but neither $2\omega_2 + \omega_3$ nor $\omega_2 + 2\omega_3$ is.
In the infinite families $\omega_1 + \omega_2, \omega_1 + \omega_{n-2}, \omega_1 + \omega_{n-1}$, we may again apply \cite[Proposition 15.25(i)]{FuHa:91} to see that all of $2\omega_1 + \omega_2, \omega_1 + 2\omega_2, 2\omega_1 + \omega_{n-2}, \omega_1 + 2\omega_{n-2}$ are not small enough.
The only remaining case is $2\omega_1 + \omega_{n-1}$, which is small enough for all $n$.
Finally, we verify that neither of $3\omega_1 + \omega_{n-1}, 2(\omega_1 + \omega_{n-1})$ are small enough, and so this list of two-supported small enough highest weights is complete.
\end{proof}

\begin{lemma}
\label{l:support-3-sl}
The only small enough 3-supported highest weight of $\SL_n$ is $\omega_1 + \omega_2 + \omega_3$ for $n = 4$.
\end{lemma}
\begin{proof}
A small enough 3-supported highest weight $\omega$ must have width exactly three, since there is no small enough 2-supported weight with width larger than 3.
Thus, $\omega = \omega_i + \omega_j + \omega_k$ for some $i,j,k$, and since one of $\omega_{j \to i}, \omega_{j \to k}$ must be small enough, we have either $\omega_{j \to i} = 2\omega_1 + \omega_{n-1}$ or $\omega_{j \to k} = 2\omega_1 + \omega_{n-1}$.
Suppose it is the former without loss of generality.
Then $i = 1$ and $k = n-1$, so $\omega = \omega_1 + \omega_j + \omega_{n-1}$.
Note that $\Gamma_\omega$ is contained in $\Gamma_{\omega_1 + \omega_j} \otimes V^*$.
The exact decomposition is given by \cite[Proposition 15.25(ii)]{FuHa:91}: we conclude that for $j = 2$,
\[ \dim \Gamma_\omega = n\left(n\binom{n}{2} - \binom{n}{3}\right) - \binom{n+1}{2} - \binom{n}{2},\]
and so $\Gamma_\omega$ is small enough in this case only when $n = 4$ and $\omega = \omega_1 + \omega_2 + \omega_3$.
When $2 < j < n-1$, there are only finitely many $(j,n)$ for which $\omega_1 + \omega_j$ is small enough, and we may check that none of them yield small enough $\omega_1 + \omega_j + \omega_{n-1}$.
\end{proof}

\begin{proposition}
\label{l:support-4-sl}
No $k$-supported highest weight of $\SL_n$ is small enough for $k > 3$, and hence the lists of the previous three lemmas are complete.
\end{proposition}
\begin{proof}
If $\omega$ is $k$-supported for $k > 3$, then $\omega$ has width at least $k$.
By applying $k-3$ shift operations in every possible way, we get a list of 3-supported weights $\alpha_1, \ldots, \alpha_m$, all of width at least $k$, such that
\[ \dim \Gamma_\omega \ge \min_i \dim \Gamma_{\alpha_i} \]
by Lemma \ref{l:shiftdim}.
But the only small enough weights of support 3 have width 3, so none of the $\alpha_i$ are small enough, and hence $\omega$ is not small enough.
\end{proof}

\vspace{1em}
The cases of $\Sp_{2n}$, $\Spin_{2n+1}$, and $\Spin_{2n}$ proceed in the same fashion as $\SL_n$ and are, in fact, easier to handle so their proofs are omitted.

\subsection{The case of $\Sp_{2n}$}

\begin{table}[h!]
\begin{tabular}{|l|l|l|l|}
\hline 
& Highest weight & Restrictions on $n$ & Purity/Cofreeness \\ \hline
1 & $\omega_1$ & $n \ge 3$ & Unstable \\
2 & $2\omega_1$ & $n \ge 3$ & Cofree \\
3 & $3\omega_1$ & $n \ge 3$ & Impure, Lemma \ref{l:sp2n-extra} \\
4 & $\omega_2$ & $n \ge 3$ & Cofree \\
5 & $\omega_3$ & $n \ge 4$ & Impure, Lemma \ref{l:sp2n-extra} \\
6 & $\omega_4$ & $n = 4$ & Cofree \\
7 & $\omega_4$ & $5 \le n \le 6$ & Impure \\
8 & $\omega_5$ & $n = 5$ & Impure \\
9 & $\omega_6$ & $n = 6$ & Impure \\
10 & $\omega_1 + \omega_3$ & $n = 3$ & Impure \\
11 & $\omega_1 + \omega_2$ & $3 \le n \le 4$ & Impure \\
\hline
\end{tabular}

\caption{Small enough representations of $\Sp_{2n}$}
\label{t:sp2n}
\end{table}
Listed in Table \ref{t:sp2n} are the small enough representations of $\Sp_{2n}$ for $n \ge 3$.
A representation $V$ of $\Sp_{2n}$ is small enough if $\dim V \le \kappa(\Sp_{2n}) = 2(n^3 + n^2 + n + 1)+1$.

\begin{proposition}
Table \ref{t:sp2n} is a complete list of the highest weights of the small enough irreducible rational representations of the symplectic groups $\Sp_{2n}$ for $n \ge 3$.
\end{proposition}

\subsection{The case of $\Spin_{2n+1}$}

\begin{table}[h!]
\begin{tabular}{|l|l|l|l|}
\hline 
& Highest weight & Restrictions on $n$ & Purity/Cofreeness \\ \hline
1 & $\omega_1$ & $n \ge 2$ & Cofree \\
2 & $2\omega_1$ & $n \ge 2$ & Cofree \\
3 & $3\omega_1$ & $n \ge 2$ & Impure, \ref{l:spin-odd-extra} \\
4 & $\omega_2$ & $n \ge 3$ & Cofree (Adjoint) \\
5 & $2\omega_2$ & $n = 2$ & Cofree (Adjoint) \\
6 & $2\omega_3$ & $n = 3$ & Impure, \ref{l:spin-odd-extra} \\
7 & $\omega_3$ & $n \ge 4$ & Impure, \ref{l:spin-odd-extra} \\
8 & $\omega_n$ & $2 \le n \le 6$ & Cofree \\
9 & $\omega_7$ & $n = 7$ & Impure, Lemma \ref{l:spin15} \\
10 & $\omega_8$ & $n = 8$ & Impure, Lemma \ref{l:spin17} \\
11 & $\omega_n$ & $9 \le n \le 11$ & Impure \\
12 & $2\omega_4$ & $n = 4$ & Impure \\
13 & $3\omega_2$ & $n = 2$ & Impure \\
14 & $\omega_1 + \omega_n$ & $2 \le n \le 4$ & Impure \\
\hline
\end{tabular}
\caption{Small enough representations of $\Spin_{2n+1}$}
\label{t:spinodd}
\end{table}

Listed in Table \ref{t:spinodd} are the small enough representations of $\Spin_{2n+1}$ for $n \ge 2$.
Let $L_1, \ldots, L_n$ be the standard dual basis in $\fh^*$, so the fundamental weights are $\omega_i = L_1 + \cdots + L_i$ for $i < n$ and $\omega_n = \frac{1}{2} (L_1 + \ldots + L_n)$.
A representation $V$ of $\Spin_{2n+1}$ is small enough if $\dim V \le \kappa(\Spin_{2n+1}) = 2(n^3 + n^2 + n + 1)+1$.

\begin{proposition}
Table \ref{t:spinodd} is a complete list of the highest weights of the small enough irreducible rational representations of the odd spin groups $\Spin_{2n+1}$ for $n \ge 2$.
\end{proposition}

\subsection{The case of $\Spin_{2n}$}

\begin{table}[h!]
\begin{tabular}{|l|l|l|l|}
\hline 
& Highest weight & Restrictions on $n$ & Purity/Cofreeness \\ \hline
1& $\omega_{n-1}, \omega_n$ & $n=5,6,7,8$ &  Cofree \\
2& $\omega_8, \omega_9$ & $n = 9$ & Impure, Lemma \ref{l:spin18} \\
3& $\omega_1$ & $n \ge 4$ & Cofree \\
4& $\omega_2$ & $n \ge 4$ & Cofree (Adjoint) \\
5& $\omega_3$ & $n \ge 5$ & Impure, Lemma \ref{l:spin-even-extra} \\
6& $2\omega_1$ & $n \ge 4$ & Cofree \\
7& $3\omega_1$ & $n \ge 4$ & Impure, Lemma \ref{l:spin-even-extra} \\
8& $\omega_3, \omega_4$ & $n = 4$ &  Cofree, Remark \ref{r:outerauto} \\
9& $2\omega_3, 2\omega_4$ & $n = 4$ & Cofree, Remark \ref{r:outerauto} \\
10& $3\omega_3, 3\omega_4$ & $n = 4$ & Impure \\
11& $2\omega_4, 2\omega_5$ & $n = 5$ & Unstable \\
12& $\omega_{n-1} + \omega_{n}$ & $n = 4, 5$ & Impure \\
13& $\omega_1 + \omega_{n-1}, \omega_1 + \omega_n$ & $n = 4, 5, 6$ &Impure \\
\hline
\end{tabular}
\caption{Small enough representations of $\Spin_{2n}$}
\label{t:spineven}
\end{table}

Listed in Table \ref{t:spineven} are the small enough representations of $\Spin_{2n}$ for $n \ge 4$.
Let $L_1, \ldots, L_n$ be the standard basis in $\fh^*$, and then we have the fundamental weights $\omega_i = L_1 + \cdots + L_i$ for $i \le n-2$, and $\omega_{n-1} = \frac{1}{2}(L_1 + \ldots + L_{n-1} - L_n)$, $\omega_n = \frac{1}{2} (L_1 + \ldots + L_n)$.
A representation of $\Spin_{2n}$ is small enough if $\dim V \le \kappa(\Spin_{2n}) = 2(n^3 + n + 3/2)+1$ for $n > 4$, or $\dim V \le 139$ when $n = 4$.

\begin{proposition}
Table \ref{t:spineven} is a complete list of the highest weights of the small enough irreducible rational representations of the odd spin groups $\Spin_{2n}$ for $n \ge 4$.
\end{proposition}

\subsection{The exceptional groups}
\begin{proposition}
Theorem \ref{thm:simple} holds for all exceptional groups.
\end{proposition}
\begin{proof}
A computer check reveals that any highest weight for an exceptional group either has dimension greater than the bound of Proposition \ref{prop:bounddimV}, is the adjoint representation, the smallest nontrivial representation, or is the representation $2\omega_1$ of $G_2$.
The adjoints and smallest non-trivial representations are cofree for all the exceptional groups.
If $V$ is the $G_2$-representation $2\omega_1$, one checks that every order 3 toral element $t$ has $\dim V^t = 9 \le \dim V - \dim G_2 - 2 = 27 - 14 - 2 = 11$, and so $2\omega_1$ is also not pure by Lemma \ref{l:toral}. 
Thus pure implies cofree for irreducible representations of exceptional groups.
\end{proof}

\section{Remaining cases}
\label{sec:notpure}


In this section we complete the proof of Theorem \ref{thm:simple} by demonstrating that the representations in Tables 1--4 that were not cofree are not pure.
By Lemma \ref{l:toral}, it suffices to check that $\dim V^t \sslash N_t$ is at most $\dim V - \dim G - 2$ for all toral elements of any fixed order, up to the action of the Weyl group, which we will do for every representation but the spinor representation of $\Spin_{15}$.

\subsection{The case of $\SL_n$}

\begin{lemma}
\label{l:sl-omega3}
The representation $\Gamma_{\omega_3} = \extp^3 \SL_n$ is not pure for $n \ge 10$.
\end{lemma}
\begin{proof}
Let $t$ be a toral element of order $2$ and let $k = \dim \SL_n^t$; note that $n-k$ is always even.
%

As a representation of $G':=\SL_k\times\SL_{n-k}$, we have,
\[
\left(\extp^3\SL_n\right)^t=\extp^3\SL_k\, \oplus\ \left(\SL_k\otimes\extp^2\SL_{n-k}\right).
\]
Let $H$ be the stabilizer of a generic point for the $G'$-action on $(\extp^3\SL_n)^t$.
Since $G'$ normalizes $t$, from Lemma \ref{l:toral} it follows that if
\begin{equation}\label{eqn:toral-lemma}
\dim (\extp^3\SL_n)^t \sslash G'\ \le\ \dim \extp^3 \SL_n - \dim \SL_n - 2\ =\ \binom{n}{3} - (n^2-1) - 2,
\end{equation}
then $\extp^3\SL_n$ is not a pure representation of $\SL_n$.

Note that $H\subseteq H'\times\SL_{n-k}$, where $H'$ is the stabilizer of a generic point for the action of $\SL_k$ on $\extp^3\SL_k$. For $k\geq10$, this representation is properly stable, so $\dim H' = 0$; for $k\leq9$, $H'$ is listed in the generic stabilizer table of \cite[p.~261]{PoVi:94}.
So
\begin{align*}
\dim\left(\left(\extp^3\SL_n\right)^t\sslash G'\right) &\leq \dim\left(\extp^3\SL_n\right)^t-(k^2-1)+\dim H'\\
&=\binom{k}{3} + k\binom{n-k}{2} - (k^2-1) + \dim H'.
\end{align*}
Comparing with the previous equation, it is enough to show
\begin{equation}\label{eqn:wedge3-ineq}
\dim H' \le \binom{n}{3} - (n^2-1) - 2 - \binom{k}{3} - k\binom{n-k}{2} + (k^2-1).
\end{equation}
The table of \cite[p.~261]{PoVi:94} reveals that $\dim H' \le 16$ always, and as a result we can compute that if $n \ge 15$, then (\ref{eqn:wedge3-ineq}) holds for any $k$. When $10\leq n\leq 14$, one checks that (\ref{eqn:wedge3-ineq}) holds for all $(n,k)$ except for
\[
(10, 2), (10, 4), (10, 6), (11,1), (11,3), (11,5), (12,2), (12,4), (12,6), (13,3), (13,5), (14, 4), (14,6);
\]
for example, if $k=7$, then \cite[p.~261]{PoVi:94} shows that $\dim H'=14$ and (\ref{eqn:wedge3-ineq}) holds for $n\in\{11,13\}$.


It remains to handle the $(n,k)$ pairs listed above. We begin with the case where $k=1$. Then $(\extp^3\SL_n)^t \sslash G' =\extp^2\SL_{n-1} \sslash\SL_{n-1}$ which is $1$-dimensional, hence (\ref{eqn:toral-lemma}) holds.

When $k=2$, our $G'$-representation $(\extp^3\SL_n)^t$ is $\SL_2\otimes\extp^2\SL_{n-2}$. From \cite[Table 6]{Elash:72}, we see the dimension of the generic stabilizer is $\frac{3}{2}(n-2)$, so
\[
\dim\left(\left(\extp^3\SL_n\right)^t\sslash G'\right)\leq 3{n-2\choose 2}-\dim G'+\frac{3}{2}(n-2)\leq \binom{n}{3} - (n^2-1) - 2,
\]
again showing (\ref{eqn:toral-lemma}) holds.

With the exception of $(n,k)=(10,6)$, for all remaining $(n,k)$, the $G'$-representation $\SL_k\otimes\extp^2\SL_{n-k}$ falls within Case 6 of \cite{Elash:72}, but does not appear on Table 6 of (loc.~cit.), so $\dim H=0$.
It follows that 
\[
\dim\left(\left(\extp^3\SL_n\right)^t\sslash G'\right) \leq \binom{k}{3}+k\binom{n-k}{2}-\dim G'\leq \binom{n}{3} - (n^2-1) - 2,
\]
and so (\ref{eqn:toral-lemma}) holds.


We now turn to the last case: $(n,k)=(10,6)$. We must show that the quotient of $V := (\SL_6 \otimes \extp^2 \SL_4) \oplus \extp^3 \SL_6$ by the normalizer $N\subseteq\SL_{10}$ of $t=\diag(1,1,1,1,1,1,-1,-1,-1,-1)$ has dimension at most $\dim\extp^3\SL_{10}-\dim\SL_{10}-2=19$.

First, note that $N$ contains the image of the map $\GL_6 \times \SL_4 \to \SL_{10}$ given by $(g,h) \mapsto \diag(g, \det{g}^{-1}h)$.
The kernel of this map is finite, so $\dim V\sslash N \le \dim V\sslash (\GL_6 \times \SL_4)$.
We will show that the latter quotient has dimension at most 19.

To simplify notation, let $G = \GL_6$ and $H = \SL_4$. Note that $(V \oplus W)/(G \times H) = ((V \oplus W)/G)/H$.
As a $G= \GL_6$-module, $V = \SL_6^{\oplus 6}$, and $G$ acts generically freely on $V$; in particular, the quotient is 0-dimensional.
Thus $\dim (V +W) /G = \dim V/G + \dim W = 20$.
To prove the assertion, we need to show that $H$ acts non-trivially on this quotient and that the generic orbit is closed.
The action of $h \in H$ on a $G$-orbit $\overline{(v,w)}$ is given by $h\overline{(v,w)} = \overline{(hv, hw)} = \overline{(hv, w)}$, where the last equality is because $H = \SL_4$ acts trivially on $W = \extp^3 \SL_6$.

If $v \in V$ is a generic vector in $V$ then $\stab_G v = 1$.
In addition $H$ acts trivially on $V/G$ (as it is 0-dimensional), so if $h \in H$, there is a unique $g \in V$ such that $hv = gv$.
This defines a morphism $\varphi_v \colon H \to G$, and we can rewrite the action of $H$ on $(V + W)/G$ as $h\overline{(v,w)} = \overline{(hv, w)} = \overline{(v, \varphi(h)^{-1}w)}$.

The map $\varphi_v$ must necessarily have finite kernel, because $H= \SL_4$ is a simple group and the image is non-trivial---$H$ acts non-trivially on $V$, so we can pick $v$ such that $\stab_v H \neq H$.
Thus $h$ stabilizes $\overline{(v,w)}$ if and only $\varphi_v(h) \subset \stab_Gw$.
Since $w$ is independent of $v$, we claim that there must exist $w \in W$ such that the image of $\varphi_v(H)$ is not contained in $\stab_G w$.
This would mean that it is a strictly smaller dimensional subgroup of $\varphi_v(H)$, since a proper subgroup of a connected group always has strictly smaller dimension.

To prove the claim, suppose to the contrary that $\varphi_v(H) \subset \stab_G w$ for all $w$ in a dense open set.
Then $\varphi_v(H)$ is a subgroup of the kernel of the $G$ action on $\extp^3 \SL_6$. But the kernel of this action is trivial. 
\end{proof}

\begin{lemma}
\label{l:sln-extra}
For $n\geq4$, the representations $\Gamma_{\omega_1 + \omega_2}$, $\Gamma_{2\omega_1 + \omega_{n-1}}$, $\Gamma_{\omega_1 + \omega_{n-2}}$, and $\Gamma_{3\omega_1}$ are not pure.
\end{lemma}
\begin{proof}
We first consider the case of $V:=\Gamma_{\omega_1 + \omega_2}$. 
From the decomposition
\[ V\oplus \extp^3 \SL_n \simeq \SL_n \otimes \extp^2 \SL_n \]
we see that if $t=\diag(1,\dots,1,-1,\dots,-1)$ is an order 2 toral element with $k$ ones, then
\[
\dim V^t = \dim \left(\SL_n\otimes\extp^2 \SL_n\right)^t - \dim \left(\extp^3 \SL_n\right)^t
=k(n-k)^2+k\binom{k}{2}-\binom{k}{3}.
\]
Since $t\neq1$, we have $0\leq k\leq n-2$. With these constraints, one checks that the above cubic is maximized at $k=n-2$, so for $n\geq7$, we have
\[
\dim V^t \leq n\binom{n}{2}-\binom{n}{3} - (n^2-1)-2= \dim V - \dim \SL_n - 2
\]
showing that $V$ is not pure by Lemma \ref{l:toral}.

The cases $4\leq n\leq 6$ are handled as follows. 
When $n = 6$, one may use LiE to directly verify that $\dim V^t \le \dim V - \dim \SL_6 - 2$ for all $t$ in a set of toral representatives of order 3, and so $V$ is not pure by Lemma \ref{l:toral}.
When $n = 5$, one check by computer that the maximum number of weights lying on a hyperplane in the weight space for $V$ is $14<15=\dim V-\dim \SL_5 - 1$, so by Proposition \ref{prop:hyperplanes} it is not pure.

For $n = 4$, we check using toral elements of order four that we always have $\dim V^t \sslash N_t \le 3$, branching the representation as an $\SL_2 \times \SL_2$ module as necessary.
We are grateful for the referee for suggesting a simpler proof of this case.
There are five toral elements of order four up to Weyl group conjugacy, of which two have no fixed points, and the remaining three are the elements $\diag(i, -i, -1, -1), \diag(i, i, -1, 1), \diag(i, -i, 1, 1)$ where $i$ is a primitive fourth root of unity.
Take for example $t = \diag(i, -i, 1, 1)$ which is normalized by $\{1\} \times \SL_2$.
As a representation of $\SL_2 \times \SL_2$, we have $V = \CC^2 \otimes \SL_2 \oplus \SL_2 \otimes \CC^2 \oplus \SL_2 \otimes \Sym^2 \SL_2 \oplus \Sym^2 \SL_2 \otimes \SL_2$, and so $V^t = \CC^2 \otimes \SL_2 \oplus (\Sym^2 \SL_2)^0 \otimes \SL_2$, where $(\Sym^2 \SL_2)^0$ denotes the $0$ weight space of $\Sym^2 \SL_2$.
From here we can see that $\dim V^t \sslash N_t \le \dim V^t / \{1\} \times \SL_2 \le 3$.
A similar analysis can be performed for the other two toral elements, showing that $V$ is not pure for $n=4$.

The cases $\Gamma_{2\omega_1 +\omega_{n-1}}$ and $\Gamma_{\omega_1+\omega_{n-2}}$ follow similarly using the decompositions $\Gamma_{2\omega_1 +\omega_{n-1}} \oplus \SL_n \simeq \SL_n\otimes\Sym^2 \SL_n$ and $\SL_n \otimes \extp^{n-2}\SL_n \simeq \SL_n^* \oplus \Gamma_{\omega_1+\omega_{n-2}}$. In both cases, $\dim V - \dim \SL_n - 2-\dim V^t$ are again cubics in $\dim \SL_n^t$.
In the former case, this cubic is always non-negative for $n > 6$, meaning that the representation is not pure for $n \ge 6$ by Lemma \ref{l:toral}.
In the latter case, this cubic is always non-negative for $n>6$ and $1\leq \dim \SL_n^t \leq n$, and in the last case where $\dim \SL_n^t = 0$, then $t = -I$.
The remaining for both representations where $n=4,5,6$ are handled with LiE using toral elements of order 3 and using Lemma \ref{l:toral}.

Lastly, for $V:=\Gamma_{3\omega_1} \simeq \Sym^3 \SL_n$, we have $V^t=\Sym^3\SL_k\oplus(\SL_k\otimes\Sym^2\SL_{n-k})$.
When $n \ge 7$, then one can check $\dim V^t \le \dim V - \dim \SL_n - 2$ for all $k$ except $k = n$ when $n$ is even; but then $t = -I$.
For $4 \le n \le 6$ we adopt a different approach.
It is proved in \cite{MFK:94} that every smooth cubic hypersurface of degree 3 is GIT stable.
This implies that if $V = \Sym^3 \SL_n$ then  $V^{\sss}$ is contained in the discriminant divisor.
For $n=4,5,6$ (cubic surfaces, threefolds and fourfolds) work in GIT \cite{ACT:02, All:03, Laz:09} implies that the generic singular hypersurface is GIT stable. 
Since the discriminant divisor is irreducible this implies that $V^{\sss}$ cannot have codimension one since it is a proper algebraic subset of the discriminant.
\end{proof}

\subsection{The case of $\Sp_{2n}$}

\begin{lemma}
\label{l:sp2n-extra}
The $\Sp_{2n}$-representations $\Gamma_{\omega_3}$ and $\Gamma_{3\omega_1}$ for $n \ge 5$ are not pure.
\end{lemma}
\begin{proof}
The representation $\Gamma_{\omega_3}$ is the kernel of the contraction $\extp^3 \Sp_{2n} \to \Sp_{2n}$ using the bilinear form, of dimension $\binom{2n}{3} - 2n$.
Let $t$ be a toral element of order 2, and let $k = \dim \Sp_{2n}^t$.
Note that since with respect to the standard maximal torus, $t = \diag(x_1, \ldots, x_n, x_1^{-1}, \ldots, x_n^{-1})$ for some $x_i = \pm 1$, it follows that $k$ is even.
By decomposing $\Sp_{2n}$ into its positive and negative eigenspaces under $t$ and expanding the exterior product as a sum of tensors, we find by the exactness of taking invariants that
\[
\dim \Gamma_{\omega_3}^t = \binom{k}{3} + k\binom{2n-k}{2} - k.
\]
For $0 \le k \le 2n-2$, this cubic is maximized at $k = 2n-2$ for $n \ge 6$, and if $k > 2n-2$, then since $k$ must be even, in fact $k = 2n$ and $t = 1$ is not of order 2.
Thus we have an upper bound $\dim \Gamma_{\omega_3}^t \le \binom{2n-2}{3}$ for all $n \ge 6$.
We then check that $\binom{2n-2}{3} \le \dim \Gamma_{\omega_3} - \dim\Sp_{2n} - 2$ for $n \geq 6$, showing that $\Gamma_{\omega_3}$ is not pure for such $n$ by Lemma \ref{l:toral}. 

For $n = 5$, we find that $\dim V^t \le \dim V - \dim\Sp_{2n} - 2$ for toral elements of order 3 except $t = \diag(\zeta,1,1,1,1,\zeta^2,1,1,1,1)$, up to Weyl group conjugacy.
For this $t$, we have
\[ \Gamma_{\omega_3}^t=\left(\ker \extp^3 \Sp_{10} \to \Sp_{10}\right)^t \simeq \ker \left(\left(\extp^3 \Sp_{10}\right)^t \to \Sp_{10}^t\right) \simeq \Sp_8 \oplus \extp^3 \Sp_8 \]
by the exactness of taking invariants.
Since $\Sp_8 \times \Sp_2 \subseteq N_t$, we find
\[ \dim \Gamma_{\omega_3}^t \sslash N_t \le \dim \left(\extp^3 \Sp_8\right) \sslash (\Sp_8 \times \Sp_2) = \dim \left(\extp^3 \Sp_8\right) \sslash \Sp_8. \]
However, since $\Sp_8 \oplus \extp^3 \Sp_8$ is a reducible representation of a simple Lie group that does not appear on the table of \cite{ElashReducible:72}, it follows that $\Sp_8$ acts with trivial generic stabilizer on this representation, and so
\[ \dim \Gamma_{\omega_3}^t \sslash N_t \le \dim \extp^3 \Sp_8 - \dim \Sp_8 = 64 - 36 = 28 \le 53 = \dim V - \dim \Sp_{10} - 2, \]
and the representation is not pure by Lemma \ref{l:toral}.

The proof for $\Gamma_{3\omega_1} = \Sym^3 \Sp_{2n}$ is similar: let $t$ be a toral element of order 2, and let $k = \dim \Sp_{2n}^t$.
Then we compute
\[ \dim \Gamma_{3\omega_1}^t = \binom{k+2}{3} + k\binom{2n-k+1}{2}, \]
and find again that on $0 \le k \le 2n-2$, it is maximized at $k= 2n-2$.
So $\dim \Gamma_{3\omega_1}^t \le \binom{2n}{3}$, and we check that $\binom{2n}{3} \le \dim \Gamma_{3\omega_1} - \dim \Sp_{2n} - 2$ for all $n \ge 5$, so these representations are not pure by Lemma \ref{l:toral}.
\end{proof}

\subsection{The case of $\Spin_{2n+1}$}

\begin{lemma}
\label{l:spin-odd-extra}
The representations $\Gamma_{\omega_3}$ for $n \ge 4$ and $\Gamma_{3\omega_1}$ for $n \ge 2$ of $\Spin_{2n+1}$ are not pure.
\end{lemma}
\begin{proof}
The representation $\Gamma_{3\omega_1}$ has dimension $\binom{2n+1+2}{3} - (2n+1)$, and is the kernel of the contraction $\Sym^3 \Spin_{2n+1} \to \Spin_{2n+1}$ by the symmetric bilinear form.
For any toral element $t$, exactness of taking invariants gives
\[ \dim \left(\ker \Sym^3 \Spin_{2n+1} \to \Spin_{2n+1}\right)^t = \dim (\Sym^3 \Spin_{2n+1})^t - \dim \Spin_{2n+1}^t,\]
so if $t$ is of order 2 and $k = \dim \Spin_{2n+1}^t$, then
\[ \dim \Gamma_{3 \omega_1}^t = \binom{k+2}{3} + k\binom{2n-k+2}{2} - k\]
If $k = 2n+1$ then $t = 1$ is not of order 2.
Otherwise, note that an order two toral element acting on $\Spin_{2n+1}$ always has odd-dimensional positive eigenspace; as a matrix, it can be taken to be of the form $\diag(x_1, \ldots, x_n, x_1^{-1}, \ldots, x_n^{-1}, 1)$ in the standard maximal torus, where each $x_i = \pm 1$, and so the number of entries with value 1 is always odd.

One may check that as a function of $k$ on the range $[0, 2n-1]$, $\dim \Gamma_{3 \omega_1}^t$ is maximized at $k = 2n-1$.
This gives us an upper bound $\dim \Gamma_{3 \omega_1}^t \le \binom{2n+1}{3}$, and we may check that $\binom{2n+1}{3} \le \dim \Gamma_{3\omega_1} - \dim \Spin_{2n+1} - 2$ always, so by Lemma \ref{l:toral}, $\Gamma_{3\omega_1}$ is not a pure representation of $\Spin_{2n+1}$.

For $\Gamma_{\omega_3}$, we note that it is equal to the third exterior power $\extp^3 \Spin_{2n+1}$ for $n \ge 4$. 
As a result, if $t$ is a toral element of order 2, and $k = \dim \Spin_{2n+1}^t$, then
\[ \dim \Gamma_{\omega_3}^t = \binom{k}{3} + k\binom{2n + 1 - k}{2}, \]
and it can be checked that this is less than $\dim \Gamma_{\omega_3} - \dim \Spin_{2n+1} - 2$ for all $(n,k)$ when $n \ge 4$ and $0 \le k \le 2n+1$ except when $k \ge 2n$.
As we remarked above, since $k$ must be odd, $k \ge 2n$ implies $k = 2n+1$ and then $t = 1$ is not of order 2, and thus by Lemma \ref{l:toral}, $\Gamma_{\omega_3}$ is not pure for any $n \ge 4$.
\end{proof}

%
%

\begin{lemma}
\label{l:spin15}
The spinor representation of $\Spin_{15}$ is not pure.
\end{lemma}
\begin{proof}
Let $V=\Spin^+_{16}$ be the positive half-spinor representation, which is a cofree representation of $\Spin_{16}$, hence pure by Proposition \ref{prop:polar-reps-Vsss-div}. Then the spinor representation $\Spin_{15}$ is the restriction of $V$ under the natural inclusion $\Spin_{15}\subset\Spin_{16}$ of Lie groups. To show $\Spin_{15}$ is not pure, we employ the following strategy. Since $\Spin_{15}$ and $\Spin^+_{16}$ are properly stable representations, by the Hilbert--Mumford Criterion, we have an inclusion
\[ V^{sss}(\Spin_{15})\subseteq V^{sss}(\Spin_{16}). \]
We know that $V^{sss}(\Spin_{16})$ is pure of codimension $1$, so to prove $\Spin_{15}$ is not pure, it suffices to show the above inclusion is strict.

To do so, consider the set
\[ \mathcal{S}=\{\{1, 2, 3, 4, 5, 6, 7\}, \{1, 2, 3\}, \{1, 4, 5\}, \{1, 6, 7\}, \{2, 4, 6\}, \{2, 5, 7\}, \{3, 4, 7\}, \{3, 5, 6\}\}. \]
For each $S\in\mathcal{S}$, let $a_{S,i}=1$ if $i\in S$ and $a_{S,i}=-1$ if $i\notin S$. We then obtain a set of $8$ weights
\[ \mathcal{W}=\left\{\sum_{i=1}^7\frac{a_{S,i}}{2}L_i+\frac{1}{2}L_8\,\middle|\, S\in\mathcal{S}\right\} \]
for $\Spin^+_{16}$. For each $\mu\in\mathcal{W}$, choose a non-zero weight vector $v_\mu$ and let $v=\sum_{\mu}v_\mu$. Since every $\mu$ has a positive $L_8$-coefficient, $v\in V^{sss}(\Spin_{16})$.

We claim that $v\notin V^{sss}(\Spin_{15})$. To see this, let $\mu'$ be the weight of $\Spin_{15}$ induced by $\mu$. Then when $V$ is viewed as a $\Spin_{15}$-representation, our set of weights $\mathcal{W}$ map to the set
\[ \mathcal{W}'=\left\{\frac{1}{2}\sum_{i=1}^7a_{S,i}L_i\,\middle|\, S\in\mathcal{S}\right\}, \]
and moreover, $\sum_{\mu'\in\mathcal{W}'}\mu'=0$. Therefore, the origin is the interior of the convex hull of $\mathcal{W}'$. We further see that the weights of $\mathcal{W}'$ satisfy a strong orthogonality relation: any two $\mu'_1,\mu'_2\in\mathcal{W}'$ differ by $4$ sign flips, so $\mu'_1-\mu'_2$ is not a root. It follows from \cite[Proposition 1.2]{DaKa:85} that $v\in V^s(\Spin_{15})$, the complement of $V^{sss}(\Spin_{15})$.
\end{proof}
\begin{remark}
Lemma \ref{l:spin15} did not rely on Proposition \ref{prop:hyperplanes} to show that the spinor representation of $\Spin_{15}$ was not pure, and so the content of Remark \ref{rem:duals} does not apply.
However, since this representation is self-dual, the dual representation does not need to be handled differently.
\end{remark}

\begin{lemma}
\label{l:spin17}
The spinor representation of $\Spin_{17}$ is not pure.
\end{lemma}
\begin{proof}
Let $V=\Spin_{17}$ be the spinor representation of the group $\Spin_{17}$ and
let $t=\diag(\zeta^{m_1},\dots,\zeta^{m_8})$ be a toral element with $\zeta$ a primitive 3rd root of unity. Using the fact that $t$ acts trivially on the weight space of $\frac{1}{2}\sum_{i=1}^8 a_iL_i$ if and only if $\sum_{i=1}^n a_i m_i = 0 \mod 3$, one checks that $\dim V^t\leq 118=2^8-{17\choose 2}-2$ unless $t=\diag(1,\dots,1,\zeta,\zeta^m)$ with $m=\pm1$.

To handle this remaining case, we begin by constructing a copy of $\Spin_{13}$ in the centralizer of $t$.
In the notation of \cite[p.~370 (23.8)]{FuHa:91}, we have
\[ t = w(1,\dots,1,\zeta,\zeta^m)=\frac{1}{4}(\zeta e_7e_{15}+\zeta^{-1}e_{15}e_7)(\zeta^m e_8e_{16}+\zeta^{-m}e_{16}e_8). \]
Let $\mathcal{S}=\{7,8,15,16\}$. Since $e_i$ is orthogonal to $e_j$ for $i\notin\mathcal{S}$ and $j\in\mathcal{S}$, the even Clifford algebra generated by the $e_i$ for $i\notin\mathcal{S}$ yields a copy of $\Spin_{13}\subset\Spin_{17}$ that commutes with $t$.

Next, since $V^t$ is the direct sum of the weight spaces with weights $\frac{1}{2}\sum_{i=1}^8 a_iL_i$ and $a_7+ma_8=0\mod 3$, 
as a $\Spin_{13}$-representation, we have
\[ V^t=(\Spin_{13})^{\oplus2}, \]
i.e., it is two copies of the spinor representation. From \cite[p.~262]{PoVi:94}, we see the generic stabilizer of $\Spin_{13}$ is $16$-dimensional, so
\[ \dim(V^t\sslash N_t)\leq\dim(V^t\sslash \Spin_{13})\leq 2^7-{13\choose 2}+16=66\leq118. \]
Thus, $V$ is not pure.
\end{proof}

\subsection{The case of $\Spin_{2n}$}

\begin{lemma}
\label{l:spin-even-extra}
The representations $\Gamma_{\omega_3}$ for $n \ge 5$ and $\Gamma_{3\omega_1}$ for $n \ge 3$ of $\Spin_{2n}$ are not pure.
\end{lemma}
\begin{proof}
This proof is very similar to the proof of Lemma \ref{l:spin-odd-extra}, so we carry it out with somewhat less detail.
Note that $\Gamma_{\omega_3} = \extp^3 \Spin_{2n}$ for $n \ge 5$, and so if $t$ is an order 2 toral element and $k = \dim \Spin_{2n}^t$, then
\[ \dim \Gamma_{\omega_3}^t = \binom{k}{3} + k\binom{2n-k}{2}. \]
As we remarked previously for the odd spin groups, the only possibilities for $k$ will be even, and if $k = 2n$ then $t = 1$ is not of order 2; so only $0 \le k \le 2n-2$ must be considered.
Except for $n = 5$, the maximum is always reached at the endpoint $2n-2$, where we can check that $\dim \Gamma_{\omega_3}^t \le \dim \Gamma_{\omega_3} - \dim \Spin_{2n} - 2$, and so none of these representations are pure by Lemma \ref{l:toral}.
When $n = 5$ we can check again that no choice of $0 \le k \le 8$ gives $\dim \Gamma_{\omega_3}^t \le \dim \Gamma_{\omega_3} - \dim \Spin_{10} - 2$, so once more by Lemma \ref{l:toral}, $\Gamma_{\omega_3}$ is not pure for $n=5$.

The proof for $\Gamma_{3\omega_1}$ is similar to the case of the third symmetric power for odd spin groups: we have $\Gamma_{3\omega_1} = \ker(\Sym^3 \Spin_{2n} \to \Spin_{2n})$ and so for a toral element $t$ of order two,
\[ \dim \Gamma_{3\omega_1}^t = \dim (\ker \Sym^3 \Spin_{2n} \to \Spin_{2n})^t = \dim (\Sym^3 \Spin_{2n})^t - \dim \Spin_{2n}^t. \]
If we write $k = \dim \Spin_{2n}^t$, then
\[ \dim \Gamma_{3\omega_1}^t = \binom{k+2}{3}+k\binom{2n-k+1}{2} - k \]
and we can again check that for $0 \le k \le 2n-2$, this is less than $\dim \Gamma_{3\omega_1} - \dim \Spin_{2n} - 2$, so by Lemma \ref{l:toral}, these representations are not pure.
\end{proof}


\begin{lemma}
\label{l:spin18}
The half-spinor representations of $\Spin_{18}$ are not pure.
\end{lemma}
\begin{proof}
Since there is an outer automorphism of $\Spin_{18}$ interchanging the two half-spinor representations, it suffices by Remark \ref{r:outerauto} to consider the representation $V=\Spin^+_{18}$. Then the weights of $V$ are $\frac{1}{2}\sum_{i=1}^9 a_iL_i$ with $a_i=\pm1$, and an even number of $a_i=-1$. We follow the same strategy of proof as in Lemma \ref{l:spin17}. One checks that if $t$ is an order 3 toral element, then $\dim V^t\leq 101=2^8-{18\choose 2}-2$ unless $t=\diag(1,\dots,1,\zeta,\zeta^m)$ with $m=\pm1$. It remains to handle this latter case.

As in the proof of Lemma \ref{l:spin17}, there is again a copy of $\Spin_{14}$ in the centralizer of $t$. Note that $V^t$ is the direct sum of the weight spaces where weights $\frac{1}{2}\sum_{i=1}^9 a_iL_i$ with $a_i=\pm1$, $a_8+ma_9=0\mod 3$, and $\prod_i a_i=1$; said another way, it is a direct sum of weight spaces with weights of the form $\frac{1}{2}(\sum_{i=7}a_iL_i + a_8(L_8-mL_9))$, where $-m=\prod_{i=1}^7a_i$. Thus, viewing $V^t$ as a $\Spin_{14}$-representation, we have
\[
V^t\simeq
\begin{cases}
(\Spin^-_{14})^{\oplus2},& m=1\\
(\Spin^+_{14})^{\oplus2},& m=-1.
\end{cases}
\]
From \cite[p.~262]{PoVi:94}, we see the generic stabilizer of $\Spin_{14}$ is $28$-dimensional, so
\[ \dim(V^t\sslash N_t)\leq\dim(V^t\sslash \Spin_{14})\leq 2^7-{14\choose 2}+28=65\leq101. \]
Thus, $V$ is not pure.
\end{proof}

\part{Actions of tori}
\label{part:torus-actions}

We now turn our attention to torus representations. We prove Theorem
\ref{part:general-facts}.\ref{thm:conj-for-tori} in
\S\ref{sec.proof-of-conj-for-tori}. In
\S\ref{sec:more-torus-examples}, we give examples that distinguish the
classes of representations pure, coprincipal, and
coregular. The most subtle of these is Example
\ref{ex:counter-example}, which shows that coprincipal is not equivalent
to pure; this is in contrast to Lemma
\ref{part:general-facts}.\ref{l:allpure} which shows that pure and coprincipal are equivalent for connected $G$ with no non-trivial characters.

\section{Proof of Theorem \ref{part:general-facts}.\ref{thm:conj-for-tori}}
\label{sec.proof-of-conj-for-tori}

Our initial goal is to prove the following proposition. This is done after several preliminary lemmas. Throughout this section, if $V$ is a $G$-representation, then we denote by $V^{\sss}(G)$ the strictly semi-stable locus for the action of $G$.

\begin{proposition} \label{prop:factor-cofree}
  Let $V_1$ and $V_2$ be stable representations
  of a torus $T$. Let $V=V_1 \oplus V_2$ be a decomposition as $T$-representations 
  and assume that $V/T  = V_1/T \times V_2/T$.
  Then $V$ is cofree (resp.~coprincipal) if and only $V_1$ and $V_2$ are cofree (coprincipal) representations.
\end{proposition}
\begin{remark}
  Note the condition that $V/T = V_1/T \times V_2/T$ is a very strong
  since it implies that $K[V]^T=K[V_1]^T \otimes_K K[V_2]^T$.
  \end{remark}

\begin{lemma} \label{lem:1Dcofree}If $V$ is a stable representation of a reductive group $G$ such that $\dim V/G = 1$, then
  $V$ is cofree and coprincipal.
\end{lemma}
\begin{proof}
Since $\dim V/G = 1$, $K[V]^G$ is a polynomial ring in one variable and hence
$K[V]$ is free over $K[V]^G$ as it is torsion free.
If $f \in K[V]^G$ generates $K[V]^G$ as a $K$-algebra, then $V(f) = V^{\sss}$,
so $V^{\sss}$ is a Cartier divisor whose image is the Cartier divisor $0 \in
\spec K[V]^G$.
\end{proof}

\begin{lemma} \label{lem:stab-prod}
Let $V_1$ and $V_2$ be representations of a reductive algebraic group $G$ and let
$V = V_1 \oplus V_2$ with the product $G \times G$ action. Then
\begin{enumerate}
\item $V^{\sss}(G \times G) = (V_1^{\sss}(G) \times V_2) \cup (V_1 \times V_2^{\sss}(G))$.
\item If $G=T$ is a torus then $V^{\sss}(T) \subset V^{\sss}(T \times T)$ where the $T$-action is the diagonal action.
\end{enumerate}
\end{lemma}
\begin{proof}
  We first show that $V^{\sss}(G \times G) \subset (V_1^{\sss}(G) \times V_2) \cup (V_1 \times V_2^{\sss}(G))$. 
  A vector $v=(v_1,v_2) \in V$ is $(G \times G)$-strictly semi-stable if and only if the orbit $(G\times G)v$ is not saturated with respect to the quotient map, i.e.~if there exists $v' = (v_1',v_2')$ such $v'$ has the same
  image in $(V_1 \oplus V_2)/(G \times G)$ and $v'$ is not in the same
  $G \times G$ orbits as $(v_1, v_2)$. Since $V_1 \oplus V_2$ has the product
  action we must have either $v_1' \notin Gv_1$ or $v_2' \notin Gv_2$.
  Assume without loss of generality that $v'_1 \notin Gv_1$.

  Since $v'$ has the same image as $v$ under the quotient map,
  $h(v) = h(v')$ for all $h \in K[V]^{G \times G} = K[V_1]^G \otimes K[V_2]^G$.
  In particular for all $f_1 \in K[V_1]^G$, $(f_1 \otimes 1)(v) = (f_1 \otimes 1)(v')$. But $(f_1 \otimes 1)(v) = f_1(v_1)$ and $(f_1 \otimes 1)(v') = f_1(v_1')$.
  Thus, $v_1$ and $v_1'$ have the same value on all $G$-invariant functions
  on $V_1$ but do not lie in the same orbit, so $v_1 \in V_1^{\sss}(G)$.

  To prove $(V_1^{\sss}(G) \times V_2) \cup (V_1 \times V_2^{\sss}(G))\subset V^{\sss}(G \times G)$, by symmetry, it enough to show $V_1^{\sss}(G) \times V_2
  \subset V^{\sss}(G \times G)$. If $v_1 \in V_1^{\sss}(G)$ then we know
  there is a vector $v'_1 \notin GV_1$ such that $f_1(v'_1) = f(v_1)$
  for all $f_1 \in K[V_1]^G$. Hence, if $v_2 \in V_2$ is any vector
  then $v = (v_1, v_2)$ and $v'=(v'_1, v_2)$ are not in the same
  $G \times G$ orbit, but $(f_1 \otimes f_2)(v) = (f_1 \times f_2)(v')$
  for all $f_1 \in K[V_1]^G$ and $f_2 \in K[V_2]^G$. Since
  $K[V]^{G \times G} = K[V_1]^G \otimes K[V_2]^G$ it follows
  that any $(G \times G)$-invariant function has the same value
  on $v$ and $v'$, but these two vectors are not in the same orbit. Hence
  $v\in V^{\sss}(G \times G)$. This proves part (1).

  We now prove (2). First note that if $V$ is any representation of a
  reductive group $G$ then a vector $v \in V$ is strictly semi-stable
  if and only if there is a vector $v' \in \overline{Gv}$ such that
  $\dim G_{v'} > d$ where $d$ is the generic stabilizer dimension of
  $V$. When $G=T$ is a torus, then for all vectors $v$, $G_v \supset
  K_0$ where $K_0$ is the kernel of the action and the generic stabilizer
  equals $K_0$. In particular if $\dim G_{v'} > d$ then $G_{v'}$
  contains a 1-parameter subgroup not contained in $K_0$.  It follows that $v
  \in V^{\sss}(T)$ if and only if the following condition holds: there
  is a 1-parameter subgroup $\lambda$ not contained in $K_0$ such that $v$ has
  only non-negative weights for the action of $\lambda$.

  Given $V = V_1 \oplus V_2$, let $K_1$ and $K_2$ be the kernels of the
  actions of $T$ on $V_1$ and $V_2$, respectively. Then the kernel of
  the diagonal action of $T$ on $V$ is $K_1 \cap K_2 \subset T$ and the
  kernel of the action of $T \times T$ is $K_1 \times K_2$. Suppose that $(v_1, v_2) \in V^{\sss}(T)$.
  Then there is a 1-parameter subgroup $\lambda$ of $T$ not contained in $K_1 \cap K_2$ such
  that $(v_1, v_2)$ has only non-negative weights with respect to the action of $\lambda$. The image
  of $\lambda$ in $T \times T$ under the diagonal embedding is not contained in $K_1 \times K_2$. Therefore
  $(v_1, v_2)$ is also in $V^{\sss}(T \times T)$.
\end{proof}
\begin{lemma} \label{lem:diag=prod-sss}
  Let $G$ be a reductive group. Suppose that $V = V_1 \oplus V_2$ and $V/G = V_1/G \times V_2/G$.
  Then $V^{\sss}(G) \supset V^{\sss}(G \times G)$ where the action of $G$ is the diagonal
  action.
\end{lemma}
\begin{proof}
  Suppose that $v= (v_1, v_2)$ is  a $(G \times G)$-strictly semistable point. By
  Lemma \ref{lem:stab-prod} we may assume without loss of generality that
  $v_1 \in V_1^{\sss}(G)$; so $v_1$ is not saturated with respect to the quotient map $V_1 \to V_1/G$. In other words there is a point $v'_1 \notin Gv_1$
  such that $f(v'_1) = f(v_1)$ for all $f \in K[V_1]^G$.

  We claim that $(v_1',v_2)$ is in the $G$-saturation of $(v_1,v_2)$, i.e.~$h(v_1',v_2) = h(v_1, v_2)$ for all $h \in K[V]^G$. To see this note
  that our assumption implies that $K[V]^G = K[V_1]^G \otimes K[V_2]^G$
  so $h \in K[V]^G$ can be expressed as $h = \sum a_i b_i$ where
  $a_i \in K[V_1]^G$ and $b_j\in K[V_2]^G$. Then $h(v_1',v_2) = \sum a_i(v'_1) b_i(v_2)
  = \sum a_i(v_1) b_i(v_2)=h(v_1, v_2)$ as claimed.

  Given the claim it follows that $(v_1, v_2)$ is not strictly semi-stable
  since $(v'_1, v_2)$ is not in the $G$-orbit of $(v_1, v_2)$.
\end{proof}

\begin{proof}[Proof of  Proposition \ref{prop:factor-cofree}]
Note that a representation $V$ of a group $G$ is cofree if and only
if $V/G$ is smooth and the quotient map $\pi\colon V \to V/G$ is flat.

If $V/G = V_1/G \times V_2/G$ then $V/G$ is smooth if and only 
$V_1/G$ and $V_2/G$ are smooth. By hypothesis the
quotient map $\pi$ factors as $\pi = \pi_1 \times \pi_2$
where $\pi_1 \colon V_1 \to V_1/G$ and $\pi_2 \colon V_2 \to V_2/G$
are corresponding quotient maps. Hence $\pi$ is flat if
and only $\pi_1$ and $\pi_2$ are flat. It follows that $V$ is cofree if and only if $V_1$ and $V_2$ are cofree.

If $G=T$ is a torus then by Lemmas \ref{lem:stab-prod} and \ref{lem:diag=prod-sss}, we know that
$V^{\sss}(T) = (V_1^{\sss} \times V_2) \cup (V_1 \times V_2^{\sss})$ 
so we see that $V^{\sss}$ is a union of divisors if and only if
$V_1^{\sss}$ and $V_2^{\sss}$ are. Hence $V^{\sss}$ is pure of codimension-one if and only if
$V_1^{\sss}$ and $V_2^{\sss}$ are pure of codimension-one.

Now if $D  = D_1 \times V_2$ is a divisor in $V^{\sss}$ then $D_1$ is a divisor
in $V_1^{\sss}$ and $\pi(D)=\pi_1(D_1) \times V_2/G$. Hence $\pi(D)$
is a Cartier divisor if and only if $\pi_1(D_1)$ is Cartier. A similar
statement holds for divisors in $V^{\sss}$ of the form $V_1 \times D_2$.
Therefore $V$ is coprincipal if and only if $V_1$ and $V_2$ are.
\end{proof}

We now come to the key proposition required to prove the Theorem \ref{part:general-facts}.\ref{thm:conj-for-tori}.
\begin{proposition} \label{prop:1D-factor}
  Let $V$ be a coprincipal representation of a torus $T$. Then there are $T$-representations $V_i$ such that $V=V_1\oplus V_2$ as $T$-representations, $V/T = V_1/T \times V_2/T$, and $V_1/T$ is one-dimensional.
\end{proposition}
\begin{proof}
  Let $x_1,\ldots,x_n$ be coordinates on $V$ diagonalizing the $T$-action.
 Any invariant $f \in K[x_1, \ldots, x_n]^T$ is necessarily
 a sum of invariant monomials, i.e.~$K[V]^T$ is generated by monomials.
 Let $f_1, \ldots , f_r$ be a minimal set of monomials that generate $K[V]^T$.
 If $r=1$ then the statement is trivial so we assume that $r\geq 2$.

 Since $T$ acts diagonally, $V^{\sss}$ is the union of linear subspaces.
 By purity, there is a divisorial component of $V^{\sss}$, which after
 reordering coordinates, we can assume is $V(x_1)$. Since $V(x_1)
 \subset V^{\sss}$ there is a non-trivial invariant function vanishing
 on $V(x_1)$. Since such a function is a polynomial in the monomials
 $f_1, \ldots , f_r$, we must have that $x_1|f_i$ for some
 $i$. After
 reordering we may assume that $x_1|f_1$.

 By assumption the image of $V(x_1)$ is Cartier. Since
 $V/T$ is an affine toric variety, $\Pic(V/T) = 0$ so the ideal
 $I = (x_1) \cap K[V]^T$ defining $\pi(V(x_1))$
 is principal. We claim that minimality of $f_1, \ldots , f_r$
 implies that $I = (f_1)$ and $x_1 \nmid f_i$ for $i \neq 1$.

 To prove the claim we argue as follows. Let $p = f_1^{a_1} \ldots
 f_r^{a_r}$ be a monomial generator of $I$. Since $f_1 \in I$
 we can write $f_1 = q f_1^{a_1} \ldots f_r^{a_r}$. Since
 this equation also holds in the polynomial ring $K[V]$ we conclude
 that either $q=1$ and $p=f_1$ or that $f_1$ can be expressed as a monomial
 in $f_2, \ldots , f_r$ which contradicts the minimality of $f_1, \ldots , f_r$.

 We now claim that
 if $x_i |f_1$ then $x_i \nmid f_k$
  for $i \neq 1$.
  To see this suppose that $x_2 | f_1$ and $x_2 | f_2$. Then the image
  of $V(x_2)$ is contained in the subvariety of $V/T = \spec K[f_1, f_2, \ldots , f_r]$ defined by the ideal $(f_1, f_2)$. Note that $f_1, f_2$
  are necessarily algebraically independent in $\spec K[f_1, f_2, \ldots, f_r]$
  because $f_1$ is the only generator divisible by $x_1$. Hence it follows
  that the image of $V(x_2)$ is not a divisor.

On the other hand, we will show $V(x_2) \subset V^{\sss}$ so by assumption on the representation $V$, we know that the image of $V(x_2)$ is a divisor. This will lead to a contradiction. If $V(x_2)$ is not in $V^{\sss}$ then $V(x_2)$ has dense intersection
  with the open set of stable points $V^s$.
  The quotient map $\pi_s \colon V^s \to V^s/G$ has constant dimensional fibers which are orbits. In particular, any $T$-invariant subvariety of $V^s$ is saturated, so the image of the $T$-invariant divisor $V(x_2) \cap V^s$ in $V^s/T$ would have codimension one.

 Given the claim we can, after reordering the coordinates, assume
 that $x_1, \ldots, x_s | f_1$ and $x_j \nmid f_1$ if $j > s$ and $x_i \nmid
 f_k$ if $i \leq s$ and $k \neq 1$. (Note that we must have $s < n$
 since $K[V]^T$ is generated by at least two invariants.)
Hence the invariant ring
 $K[V]^T$ is generated by $f_1 = x_1^{a_1} \ldots x_s^{a_s}$ with $a_i
 > 0$ and monomials $f_2,\ldots , f_r$ in the variables $x_{s+1},
 \ldots , x_n$. So we can split $V = V_1 \oplus V_2$ where $V_1$ is
 the subspace spanned by the coordinates $x_1, \ldots , x_s$ and $V_2$
 is the subspace spanned by the coordinates $x_{s+1}, \ldots ,
 x_n$. The invariant ring $K[V_1]^T$ consists of those elements of
 $K[V]^T$ that only involve $x_1,\dots,x_s$. Since these variables do
 not divide $f_2,\dots,f_r$, we know $K[V_1]^T$ is generated by
 $f_1$. Likewise, any $T$-invariant monomial in $x_{s+1}, \ldots ,
 x_n$ is a product of $f_2, \ldots , f_r$ so $K[V_2]^T$ is generated
 by $f_2, \ldots , f_r$. Since $f_1$ is algebraically independent from
 $f_2,\dots,f_r$ we have $K[V]^T=K[f_1][f_2,\dots,f_r]=K[V_1]^T\otimes
 K[V_2]^T$, i.e.~$V/T=V_1/T\times V_2/T$.
    \end{proof}

\begin{proof}[Proof of Theorem \ref{part:general-facts}.\ref{thm:conj-for-tori}]
  The theorem follows by induction on the dimension of $V/G$ and Propositions
  \ref{prop:factor-cofree} and \ref{prop:1D-factor}.
\end{proof}

\section{Further results and examples for torus actions}
\label{sec:more-torus-examples}

In this section, we give examples to illustrate how coregular, pure, and coprincipal differ.

\subsection{Example to show coregular does not imply pure}
Not surprisingly, there are stable coregular representations of tori which are not
pure. Here is a simple example.
\begin{example} \label{ex:coreg-notpure}
  Let $T= \GG_m$ act on a 3-dimensional vector space $V$  with weights $(1,-1,0)$. If we identify $K[V] = K[x,y,z]$ then $K[V]^G = K[xy,z]$ is regular, so
  $V$ is coregular. However, $V^{\sss}$ is the union of two codimension-two subspaces $V(x,z)$ and $V(y,z)$.
\end{example}

\subsection{Example to show that $V^{\sss}$ being pure of codimension-one does not imply pure}

\begin{example}\label{ex.purenotnupure}
Consider the $\GG_m^2$-action on $\AA^5$ with weights
\[
x=(2,0),\quad
y=(0,1),\quad
z=(-2,-1),\quad
u_1=(-1,0),\quad
u_2=(-1,0)
\]
This representation is stable and $V^{\sss}=V(x)\cup V(y)\cup V(z)$. One checks that
\[
K[x,y,z,u_1,u_2]^{\GG_m^2}=K[xyz, xu_1^2, xu_2^2, xu_1u_2].
\]
So, the quotient is $\AA^1$ times an $A_1$-singularity, 
hence it is not smooth but has finite quotient singularities. We see that $V$ is not pure as $V(x)$ maps to a point under the quotient map. 
\end{example}

\subsection{Example to show that pure does not imply coprincipal}
\begin{example}
\label{ex:counter-example}
Consider the action of $\GG_m^3$ on $\AA^6$ with weights
\[
u_1 = (0,1,0),\quad
u_2 = (1,-1,0),\quad
u_3 = (1,0,0),\quad
u_4 = (-1,0,0),\quad
y_1 = (0,0,1),\quad
y_2 = (-1,0,-1)
\]
We calculate the invariants. Let $H=(0,0,1)^\perp$ be a hyperplane in the character lattice tensored with $\RR$. Note that the $u_i\in H$, and that $y_1$ and $y_2$ are on opposite sides of $H$. 
As a result, every monomial invariant $y_1^{a_1}y_2^{a_2}\prod_i u_i^{b_i}$ must have $a_1=a_2$. Hence,
\[
K[u_1,u_2,u_3,u_4,y_1,y_2]^{\GG_m^3}=K[u_1,u_2,u_3,u_4,y_1y_2]^{\GG_m^2},
\]
where $\GG_m^2$ is the subtorus $\GG_m^2\times 1\subset\GG_m^3$. Said another way, $\AA^6/\GG_m^3\simeq\AA^5/\GG_m^2$, where $\GG_m^2$ acts on $\AA^5$ with weights
\[
u'_1 = (0,1),\quad
u'_2 = (1,-1),\quad
u'_3 = (1,0),\quad
u'_4 = w = (-1,0)
\]
Now notice that the weights $u'_3$ and $u'_4=w$ are contained on the line $L=(0,1)^\perp$, and that $u'_1$ and $u'_2$ live on opposite sides of $L$. So by the same reasoning as above,
\[
K[u_1,u_2,u_3,u_4,y_1y_2]^{\GG_m^2}=K[u_1u_2,u_3,u_4,y_1y_2]^{\GG_m},
\]
or said another another way, $\AA^5/\GG_m^2\simeq\AA^4/\GG_m$ where $\GG_m$ acts on $\AA^4$ with weights $1,1,-1,-1$. This quotient is the non-simplicial toric variety given by the cone over the quadratic surface. We have therefore shown
\[
K[u_1,u_2,u_3,u_4,y_1,y_2]^{\GG_m^3}=K[u_1u_2y_1y_2, u_1u_2u_4, u_3y_1y_2, u_3u_4].
\]
One checks that
\[
V^{\sss}=V(u_1)\cup V(u_2)\cup V(y_1)\cup V(y_2)
\]
and that each of these components maps to a divisor, so $V$ is pure. However, $V$ is not coprincipal since all of these components map to Weil divisors which are not Cartier, e.g.~$V(u_1)\subset\AA^6$ maps to $V(u_1u_2y_1y_2, u_1u_2u_4)\subset\AA^6/
\GG_m^3$ which is the divisor $a=b=0$ in the quotient $\spec K[a,b,c,d]/(ad-bc)$.
\end{example}

Note that by contrast if $V$ is an pure representation of a connected reductive group $G$ such that $\dim V/ G = 2$, then it follows from Kempf \cite{Kem:80} (cf.~\cite[Theorem 8.6]{PoVi:94}) that $V$ is cofree and hence coprincipal if Question \ref{q:nicely-pure} \ref{conj:nicely-pure::cofree->pure} has an affirmative answer.

\subsection{Co-orbifold and pure implies coprincipal}
The pure representation of Example
\ref{ex:counter-example} is not coprincipal but has worse than finite quotient singularities.
The following proposition shows that this is not an isolated phenomenom.
\begin{proposition}\label{prop.qcnice}
  If $V$ is an pure representation of a torus $T$ for which $V/G$
  is singular, then $V/G$ has worse than finite quotient singularities.
  \end{proposition}
\begin{proof}
  We will show that if $V$ is pure and the image has finite quotient
  singularities then it is in fact coprincipal and hence cofree by Theorem \ref{thm:conj-for-tori}.

  If $V/G$ has finite quotient singularities then any divisor on $V/G$ is $\QQ$-Cartier and the proposition follows from the following lemma.
  \end{proof}
  \begin{lemma} Let $V$ be a stable representation of a torus $T$
    and let $\pi \colon V \to V/T$ be the quotient map. 
    Let $Z$ be a divisorial component of $V^{\sss}$. Then
    the effective Weil divisor $[\pi(Z)]$ is $\QQ$-Cartier if and only
    it is Cartier.
  \end{lemma}
  \begin{proof}
    We use an argument similar to that used in the proof of Proposition \ref{prop:1D-factor}. As above choose coordinates $x_1, \ldots , x_n$
    diagonalizing the $T$ action and let $f_1, \ldots , f_r$
    be a minimal set of monomials that generate $K[V]^T$. After reordering the coordinates we may assume that $Z = V(x_1)$. The image $\pi(Z)$
    is the subvariety of $V/T$ defined by the contracted ideal
    $I = (x_1) \cap K[V]^T$. Since the map $V \to V/T$ is toric the ideal
    $I$ is generated by monomials $(p_1, \ldots , p_s)$. Since
    $\pi(Z)$ is $\QQ$-Cartier there is a monomial $p$ such that
    $\sqrt{(p)} = I$. Write $p = f_1^{a_1} \ldots f_r^{a_r}$. Since
    $p \in (x_1)$, after possibly reordering the coordinates we know
    that $x_1 |f_1^{a_1}$ and hence $x_1 | f_1$. Thus
    $f_1 \in I = \sqrt{(p)}$ so $f_1 = q f_1^{a_1} \ldots f_r^{a_r}$.
    for some monomial $q \in K[f_1,\ldots , f_r]$. Since this equation
    also holds in the UFD $K[V]$ we conclude that $q = 1$ and $p = f_1$.
    As in the proof of proposition \ref{prop:1D-factor} this implies
    that no other $x_1 \nmid f_i$ for $i \neq 1$. Hence $I$ must
    be generated by powers of $f_1$ so $I = (f_1)$ since $f_1 \in I$.
    Therefore $\pi(Z)$ is defined by a single equation as claimed.
  \end{proof}

\bibliographystyle{../arxiv_submitted/amsmath}

\bibliography{../arxiv_submitted/refs}

\def\cprime{$'$} \def\cprime{$'$} \def\cprime{$'$}
\begin{thebibliography}{vLCL}

\bibitem[All]{All:03}
Daniel Allcock, {\em The moduli space of cubic threefolds}, J. Algebraic Geom.
  \textbf{12} (2003), no.~2, 201--223.

\bibitem[ACT]{ACT:02}
Daniel Allcock, James~A. Carlson, and Domingo Toledo, {\em The complex
  hyperbolic geometry of the moduli space of cubic surfaces}, J. Algebraic
  Geom. \textbf{11} (2002), no.~4, 659--724.

\bibitem[AHR]{AHR:20}
Jarod Alper, Jack Hall, and David Rydh, {\em A {L}una \'{e}tale slice theorem
  for algebraic stacks}, Ann. of Math. (2) \textbf{191} (2020), no.~3,
  675--738.

\bibitem[Bri]{Bri:93}
Michel Brion, {\em Sur les modules de covariants}, Ann. Sci. \'{E}cole Norm.
  Sup. (4) \textbf{26} (1993), no.~1, 1--21.

\bibitem[DK]{DaKa:85}
Jiri Dadok and Victor Kac, {\em Polar representations}, J. Algebra \textbf{92}
  (1985), no.~2, 504--524.

\bibitem[ER]{EdRy:21}
Dan Edidin and David Rydh, {\em Canonical reduction of stabilizers for {A}rtin
  stacks with good moduli spaces}, Duke Math. J. \textbf{170} (2021), no.~5,
  827--880.

\bibitem[\'E1]{ElashReducible:72}
A.~G. \'Elashvili, {\em Canonical form and stationary subalgebras of points of
  general position for simple linear lie groups}, Funct. Anal. Appl. \textbf{6}
  (1972), no.~1, 44--53.

\bibitem[\'E2]{Elash:72}
\bysame, {\em Stationary subalgebras of points of the common state for
  irreducible linear lie groups}, Funct. Anal. Appl. \textbf{6} (1972), no.~2,
  139--148.

\bibitem[FH]{FuHa:91}
William Fulton and Joe Harris, {\em Representation theory}, Graduate Texts in
  Mathematics, vol. 129, Springer-Verlag, New York, 1991, A first course,
  Readings in Mathematics.

\bibitem[GGS]{GoGuSt:17}
Daniel Goldstein, Robert Guralnick, and Richard Stong, {\em A lower bound for
  the dimension of a highest weight module}, Represent. Theory \textbf{21}
  (2017), 611--625.

\bibitem[KPV]{KPV:76}
Victor~G. Kac, Vladimir~L. Popov, and Ernest~B. Vinberg, {\em Sur les groupes
  lin\'{e}aires alg\'{e}briques dont l'alg\`ebre des invariants est libre}, C.
  R. Acad. Sci. Paris S\'{e}r. A-B \textbf{283} (1976), no.~12, Ai, A875--A878.

\bibitem[Kem]{Kem:80}
George Kempf, {\em Some quotient surfaces are smooth}, Michigan Math. J.
  \textbf{27} (1980), no.~3, 295--299.

\bibitem[KKV]{KKV:89}
Friedrich Knop, Hanspeter Kraft, and Thierry Vust, {\em The {P}icard group of a
  {$G$}-variety}, Algebraische {T}ransformationsgruppen und
  {I}nvariantentheorie, DMV Sem., vol.~13, Birkh\"{a}user, Basel, 1989,
  pp.~77--87.

\bibitem[Laz]{Laz:09}
Radu Laza, {\em The moduli space of cubic fourfolds}, J. Algebraic Geom.
  \textbf{18} (2009), no.~3, 511--545.

\bibitem[Lun]{Lun:73}
Domingo Luna, {\em Slices \'etales}, Sur les groupes alg\'ebriques, Soc. Math.
  France, Paris, 1973, pp.~81--105. Bull. Soc. Math. France, Paris, M\'emoire
  33.

\bibitem[MM]{Mat:60}
Yoz\^{o} Matsushima and Akihiko Morimoto, {\em Sur certains espaces fibr\'{e}s
  holomorphes sur une vari\'{e}t\'{e} de {S}tein}, Bull. Soc. Math. France
  \textbf{88} (1960), 137--155.

\bibitem[MFK]{MFK:94}
D.~Mumford, J.~Fogarty, and F.~Kirwan, {\em Geometric invariant theory}, third
  ed., Springer-Verlag, Berlin, 1994.

\bibitem[PV]{PoVi:94}
V.~L. Popov and E.~B. Vinberg, {\em Invariant theory}, Algebraic Geometry IV,
  Encyclopedaedia of Mathematical Sciences, Springer-Verlag, 1994,
  pp.~123--278.

\bibitem[vLCL]{LiE:92}
Marc A.~A. van Leeuwen, Arjeh~M. Cohen, and Bert Lisser, {\em {LiE}, a package
  for lie group computations}, Computer Algebra Nederland, Amsterdam, 1992.

\bibitem[Weh]{Weh:92}
David~L. Wehlau, {\em A proof of the {P}opov conjecture for tori}, Proc. Amer.
  Math. Soc. \textbf{114} (1992), no.~3, 839--845.

\end{thebibliography}


\begin{thebibliography}{vLCL}

\bibitem[AG]{AdGo:79}
O.~M. Adamovich and E.~O. Golovina, {\em Simple linear {L}ie groups having a
  free algebra of invariants}, Selecta Math. Soviet. \textbf{3} (1983/84),
  no.~2, 183--220, Selected reprints.

\bibitem[All]{All:03}
Daniel Allcock, {\em The moduli space of cubic threefolds}, J. Algebraic Geom.
  \textbf{12} (2003), no.~2, 201--223.

\bibitem[ACT]{ACT:02}
Daniel Allcock, James~A. Carlson, and Domingo Toledo, {\em The complex
  hyperbolic geometry of the moduli space of cubic surfaces}, J. Algebraic
  Geom. \textbf{11} (2002), no.~4, 659--724.

\bibitem[BS1]{BhaSha:15}
Manjul Bhargava and Arul Shankar, {\em Binary quartic forms having bounded
  invariants, and the boundedness of the average rank of elliptic curves}, Ann.
  of Math. (2) \textbf{181} (2015), no.~1, 191--242.

\bibitem[BS2]{BhaSha2:15}
\bysame, {\em Ternary cubic forms having bounded invariants, and the existence
  of a positive proportion of elliptic curves having rank 0}, Ann. of Math. (2)
  \textbf{181} (2015), no.~2, 587--621.

\bibitem[Bri]{Bri:93}
Michel Brion, {\em Sur les modules de covariants}, Ann. Sci. \'{E}cole Norm.
  Sup. (4) \textbf{26} (1993), no.~1, 1--21.

\bibitem[DK]{DaKa:85}
Jiri Dadok and Victor Kac, {\em Polar representations}, J. Algebra \textbf{92}
  (1985), no.~2, 504--524.

\bibitem[\'E1]{ElashReducible:72}
A.~G. \'Elashvili, {\em Canonical form and stationary subalgebras of points of
  general position for simple linear lie groups}, Funct. Anal. Appl. \textbf{6}
  (1972), no.~1, 44--53.

\bibitem[\'E2]{Elash:72}
\bysame, {\em Stationary subalgebras of points of the common state for
  irreducible linear lie groups}, Funct. Anal. Appl. \textbf{6} (1972), no.~2,
  139--148.

\bibitem[FH]{FuHa:91}
William Fulton and Joe Harris, {\em Representation theory}, Graduate Texts in
  Mathematics, vol. 129, Springer-Verlag, New York, 1991, A first course,
  Readings in Mathematics.

\bibitem[GGS]{GoGuSt:17}
Daniel Goldstein, Robert Guralnick, and Richard Stong, {\em A lower bound for
  the dimension of a highest weight module}, Represent. Theory \textbf{21}
  (2017), 611--625.

\bibitem[KPV]{KPV:76}
Victor~G. Kac, Vladimir~L. Popov, and Ernest~B. Vinberg, {\em Sur les groupes
  lin\'{e}aires alg\'{e}briques dont l'alg\`ebre des invariants est libre}, C.
  R. Acad. Sci. Paris S\'{e}r. A-B \textbf{283} (1976), no.~12, Ai, A875--A878.

\bibitem[Kem]{Kem:80}
George Kempf, {\em Some quotient surfaces are smooth}, Michigan Math. J.
  \textbf{27} (1980), no.~3, 295--299.

\bibitem[Laz]{Laz:09}
Radu Laza, {\em The moduli space of cubic fourfolds}, J. Algebraic Geom.
  \textbf{18} (2009), no.~3, 511--545.

\bibitem[Lit]{Lit:89}
Peter Littelmann, {\em Koregul\"{a}re und \"{a}quidimensionale
  {D}arstellungen}, J. Algebra \textbf{123} (1989), no.~1, 193--222.

\bibitem[MM]{Mat:60}
Yoz\^{o} Matsushima and Akihiko Morimoto, {\em Sur certains espaces fibr\'{e}s
  holomorphes sur une vari\'{e}t\'{e} de {S}tein}, Bull. Soc. Math. France
  \textbf{88} (1960), 137--155.

\bibitem[MFK]{MFK:94}
D.~Mumford, J.~Fogarty, and F.~Kirwan, {\em Geometric invariant theory}, third
  ed., Springer-Verlag, Berlin, 1994.

\bibitem[Pop1]{Pop:87}
V.~L. Popov, {\em Modern developments in invariant theory}, Proceedings of the
  {I}nternational {C}ongress of {M}athematicians, {V}ol. 1, 2 ({B}erkeley,
  {C}alif., 1986), Amer. Math. Soc., Providence, RI, 1987, pp.~394--406.

\bibitem[Pop2]{Pop:15}
\bysame, {\em Number of components of the nullcone}, Proc. Steklov Inst. Math.
  \textbf{290} (2015), no.~1, 84--90, Published in Russian in Tr. Mat. Inst.
  Steklova {{\bf{290}}} (2015), 95--101.

\bibitem[PV]{PoVi:94}
V.~L. Popov and E.~B. Vinberg, {\em Invariant theory}, Algebraic Geometry IV,
  Encyclopedaedia of Mathematical Sciences, Springer-Verlag, 1994,
  pp.~123--278.

\bibitem[Sch1]{Sch:78a}
Gerald~W. Schwarz, {\em Representations of simple {L}ie groups with a free
  module of covariants}, Invent. Math. \textbf{50} (1978/79), no.~1, 1--12.

\bibitem[Sch2]{Sch:78b}
\bysame, {\em Representations of simple {L}ie groups with regular rings of
  invariants}, Invent. Math. \textbf{49} (1978), no.~2, 167--191.

\bibitem[vLCL]{LiE:92}
Marc A.~A. van Leeuwen, Arjeh~M. Cohen, and Bert Lisser, {\em {LiE}, a package
  for lie group computations}, Computer Algebra Nederland, Amsterdam, 1992.

\bibitem[Weh]{Weh:92}
David~L. Wehlau, {\em A proof of the {P}opov conjecture for tori}, Proc. Amer.
  Math. Soc. \textbf{114} (1992), no.~3, 839--845.

\end{thebibliography}
\end{document}
\def\cprime{$'$} \def\cprime{$'$} \def\cprime{$'$}

\end{document}